\documentclass[12pt,reqno]{amsart}

\usepackage[shortlabels]{enumitem}

\usepackage{amssymb}
\usepackage{bm}  
\usepackage[dvipsnames]{xcolor}
\usepackage%
[colorlinks=true,linkcolor=Maroon,citecolor=OliveGreen,pagebackref]
{hyperref}
\usepackage[capitalize]{cleveref}
\usepackage{url}
\usepackage[abbrev,shortalphabetic]{amsrefs}  
\usepackage{tikz}

\usepackage[left=20mm,top=20mm,right=20mm,bottom=20mm]{geometry}

\usepackage{subcaption}

\usepackage[norefs, nocites]{refcheck}

\makeatletter
\newcommand{\refcheckize}[1]{%
  \expandafter\let\csname @@\string#1\endcsname#1%
  \expandafter\DeclareRobustCommand\csname relax\string#1\endcsname[1]{%
    \csname @@\string#1\endcsname{##1}\wrtusdrf{##1}}%
  \expandafter\let\expandafter#1\csname relax\string#1\endcsname
}
\makeatother
\refcheckize{\cref}
\refcheckize{\Cref}

\newtheorem{theorem}{Theorem}[section]
\newtheorem{lemma}[theorem]{Lemma}
\newtheorem{proposition}[theorem]{Proposition}
\newtheorem{corollary}[theorem]{Corollary}

\theoremstyle{definition}
\newtheorem{remark}[theorem]{Remark}

\newtheorem{definition}[theorem]{Definition}

\numberwithin{equation}{section}

\usepackage{amssymb}
\renewcommand{\leq}{\leqslant}
\renewcommand{\le}{\leq}
\renewcommand{\geq}{\geqslant}
\renewcommand{\ge}{\geq}
\newcommand{\eps}{\epsilon}

\makeatletter
\newsavebox{\@brx}
\newcommand{\llangle}[1][]{\savebox{\@brx}{\(\m@th{#1\langle}\)}%
  \mathopen{\copy\@brx\kern-0.5\wd\@brx\usebox{\@brx}}}
\newcommand{\rrangle}[1][]{\savebox{\@brx}{\(\m@th{#1\rangle}\)}%
  \mathclose{\copy\@brx\kern-0.5\wd\@brx\usebox{\@brx}}}
\makeatother

\newcommand\br[1]{{\left(#1\right)}}
\newcommand\floor[1]{\left\lfloor{#1}\right\rfloor}
\newcommand\pfrac[2]{\br{\frac{#1}{#2}}}

\renewcommand{\subset}{\subseteq}
\renewcommand{\supset}{\supseteq}

\newcommand\opr[1]{\operatorname{#1}}

\def\CC{\mathbf{C}}
\def\ZZ{\mathbf{Z}}
\def\FF{\mathbf{F}}

\def\sL{\mathsf{L}}
\def\sH{\mathsf{H}}
\def\sB{\mathsf{B}}
\def\sF{\mathsf{F}}
\def\sQ{\mathsf{Q}}

\def\cG{\mathcal{G}}

\def\cU{\mathcal{U}}

\def\PP{\mathbf{P}}
\def\EE{\mathbf{E}}

\def\supp{\opr{supp}}

\def\tr{\opr{tr}}

\def\im{\opr{im}}

\def\rank{\opr{rank}}
\def\crank{\opr{crank}}
\def\trank{\opr{trank}}
\def\lrank{\opr{lrank}}
\def\cx{\opr{cx}}
\def\psystem{\mathfrak{P}} 

\def\Adj{\mathcal{A}}
\def\Irr{\opr{Irr}}

\usepackage{todonotes}

\begin{document}

\title{Transversals in quasirandom latin squares}

\author{Sean Eberhard}
\address{Sean Eberhard, Mathematical Sciences Research Centre, Queen's University Belfast, Belfast BT7~1NN, UK}
\email{s.eberhard@qub.ac.uk}

\author{Freddie Manners}
\address{Freddie Manners, UCSD Department of Mathematics, 9500 Gilman Drive \#0112, La Jolla CA 92093, USA}
\email{fmanners@ucsd.edu}

\author{Rudi Mrazovi\'c}
\address{Rudi Mrazovi\'c, University of Zagreb, Faculty of Science, Department of Mathematics, Zagreb, Croatia}
\email{Rudi.Mrazovic@math.hr}

\thanks{SE has received funding from the European Research Council (ERC) under the European Union’s Horizon 2020 research and innovation programme (grant agreement No. 803711) and from the Royal Society. RM is supported in part by the Croatian Science Foundation under the project UIP-2017-05-4129 (MUNHANAP). FM is supported by a Sloan Fellowship.} 

\begin{abstract}
    A transversal in an $n \times n$ latin square is a collection of $n$ entries not repeating any row, column, or symbol.
    Kwan showed that almost every $n \times n$ latin square has $\bigl((1 + o(1)) n / e^2\bigr)^n$ transversals as $n \to \infty$.
    Using a loose variant of the circle method we sharpen this to $(e^{-1/2} + o(1)) n!^2 / n^n$.
    Our method works for all latin squares satisfying a certain quasirandomness condition,
    which includes both random latin squares with high probability as well as multiplication tables of quasirandom groups.
\end{abstract}

\maketitle


\section{Introduction}

A transversal in an $n \times n$ latin square is a set of $n$ entries such that no two of them come from the same row or column or contain the same symbol.

Although there are examples of latin squares with no transversals (e.g., the multiplication table of $\ZZ/n\ZZ$ for $n$ even), it is widely believed that these are rare.
For example, a famous conjecture of Ryser claims that an $n \times n$ latin square contains a transversal whenever $n$ is odd.
In the same direction, Kwan~\cite{kwan} proved that a uniformly random $n \times n$ latin square has a transversal with probability $1-o(1)$.
Moreover, he showed that, with probability $1-o(1)$, the number of transversals is $\bigl((1-o(1))n/e^2\bigr)^n$.

In this paper we improve Kwan's result by finding the precise asymptotic of the number of transversals in a uniformly random latin square.
\begin{theorem}%
    \label{thm:main-random}
    Let $\sL$ be a uniformly random $n \times n$ latin square.
    Then $\sL$ has $\bigl(e^{-1/2}+o(1)\bigr) n!^2 / n^n$ transversals with probability $1 - o(1)$ as $n \to \infty$.
\end{theorem}

More generally, we find a (deterministic) \emph{quasirandomness condition} for latin squares which is sufficient to guarantee this same asymptotic number of transversals.
\begin{theorem}%
    \label{thm:main-quasirandom}
    There is a constant $\rho > 0$ such that  the following holds.
    Let $\sL$ be an $n \times n$ latin square which is $\Adj$-quasirandom with parameter $\rho$.
    Then $\sL$ has $\bigl(e^{-1/2}+o(1)\bigr) n!^2 / n^n$ transversals.
\end{theorem}

The precise definition of ``$\Adj$-quasirandom'' is in terms of the spectral gap of some operator associated to $\sL$: see \Cref{def:quasirandom}.
Despite the language, it is not actually obvious that a uniformly random $n \times n$ is quasirandom with high probability as $n \to \infty$, and hence that \Cref{thm:main-quasirandom} implies \Cref{thm:main-random}.
Indeed, it is incredibly delicate to prove \emph{any} statistical properties of a uniform random latin square, for a number of reasons: the exact asymptotic count of $n \times n$ latin squares is not known; the latin square property is too rigid to make local changes; and no efficient way of sampling uniform random latin squares is known.

However, using a recent result of Kwan, Sah, Sawhney, and Simkin~\cite{KSSS} we are indeed able to establish that a random latin square is $\Adj$-quasirandom with parameter $o(1)$, with high probability,
and we can thus prove \Cref{thm:main-random} as a consequence of \Cref{thm:main-quasirandom}.

\begin{theorem}%
    \label{thm:random-implies-quasirandom}
    Let $\sL$ be a uniformly random $n \times n$ latin square.
    Then $\sL$ is $\Adj$-quasirandom with parameter $o(1)$, with probability $1 - o(1)$ as $n \to \infty$.
\end{theorem}

Somewhat opposite to random latin squares are latin squares which are multiplication tables of finite groups.
In~\cite{EMM} we proved that as long as the underlying group satisfies a necessary condition to have at least one transversal, we have the count as in \Cref{thm:main-quasirandom} with an extra factor equal to the size of the group's abelianization.
For some groups, this result is implied by \Cref{thm:main-quasirandom} and the following (easy) result.

\begin{theorem}%
    \label{thm:quasirandom-groups-are-quasirandom}
    Let $G$ be a group and let $\sL_G$ be the multiplication table of $G$.
    Then $\sL_G$ is $\Adj$-quasirandom with parameter $1/D$, where $D$ is the minimal dimension of a nontrivial linear representation of $G$.
\end{theorem}

This shows that the $\Adj$-quasirandomness condition when restricted to group multiplication tables coincides with the usual notion of quasirandomness for groups due to Gowers~\cite{gowers}.
Thus together \Cref{thm:main-quasirandom,thm:quasirandom-groups-are-quasirandom} recover the main result of~\cite{EMM} for sufficiently quasirandom groups.

There appears to be no single universal definition of a ``quasirandom latin square'', in the same way that there is no single definition of a ``quasirandom set of integers''.
Instead there are various possible qualitatively inequivalent definitions, some more natural than others, and the correct choice depends on the problem at hand.
For this reason we prefer to talk about a \emph{quasirandomness condition} than about a ``definition of quasirandomness'', and we do not claim that the condition in \Cref{def:quasirandom} is necessarily the correct one for other problems.
In particular it is not directly related to the notion introduced in \cites{MR4012871, MR4456029}, since that depends on some additional structure (namely, an ordering on the set of symbols) to which our condition is oblivious.
See \Cref{sec:dense-minor} for further remarks.

\section{Outline}%
\label{sec:outline}

Our approach is analytical rather than combinatorial. Let $X$, $Y$, $Z$ be $n$-element sets of rows, columns and symbols. We identify an $n \times n$ latin square $\sL$ with a subset of $X \times Y \times Z$ satisfying the latin square property, i.e., every pair from $X \times Y$, $Y \times Z$ and $Z \times X$ is in exactly one triple from $\sL$.
We let $L^2(X), L^2(Y), L^2(Z)$ denote the spaces of complex-valued functions on $X$, $Y$, $Z$ (equipped with the standard hermitian inner product). The \emph{latin square tensor} $\Lambda = \Lambda_\sL$ is defined by
\[
    \Lambda(f, g, h) := \EE_{(x, y, z) \in \sL} f(x) g(y) h(z)
\]
for $f \in L^2(X), g \in L^2(Y), h \in L^2(Z)$.

We stress that the latin square tensor $\Lambda_\sL$ depends on $\sL$, but we will always just write $\Lambda$ for brevity.
We use the same notation for powers of $\sL$, in the following sense.
If $\sL_1$ and $\sL_2$ are latin squares then $\sL_1 \times \sL_2$ is also a latin square, where $\sL_1 \times \sL_2$ is simply the cartesian product of $\sL_1$ and $\sL_2$ as subsets of $X_1 \times Y_1 \times Z_1$ and $X_2 \times Y_2 \times Z_2$.
Accordingly the powers $\sL^m$ are latin squares of order $n^m$ for all $m \geq 0$, and if $f \in L^2(X^m), g \in L^2(Y^m), h \in L^2(Z^m)$ then we write
\[
    \Lambda(f, g, h) := \EE_{(x, y, z) \in \sL^m} f(x) g(y) h(z).
\]

Of particular interest is the latin square $\sL^n$. We write $S$ (or sometimes $S_X$ to emphasize the domain, and similarly $S_Y$ and $S_Z$) for the subset $S \subset X^n$ of all bijections $[n] \to X$.
Then one can check that the number of transversals in $\sL$ is
\[
    \Lambda(1_S, 1_S, 1_S) \frac{n^{2n}}{n!}.
\]
Our goal is therefore to show that
\begin{equation}
    \label{eq:main-Lambda-statement}
    \Lambda(1_S, 1_S, 1_S) = \bigl(e^{-1/2} + o(1)\bigr) \pfrac{n!}{n^n}^3,
\end{equation}
provided that $\sL$ satisfies an appropriate quasirandomness condition.

We approach~\eqref{eq:main-Lambda-statement} principally by studying how $1_S$ deviates from its density $n!/n^n$. We do this as follows.
For any set $A \subset [m]$, we may identify $L^2(X^A)$ with the subspace of $L^2(X^m)$ consisting of functions $X^m \to \CC$ that factor as $X^m \to X^A \to \CC$; i.e., functions $f(x_1,\dots,x_m)$ that only depend on $(x_i \colon i \in A)$.
These spaces are nested: if $A \subset B$ then $L^2(X^A) \subset L^2(X^B)$.
We write $Q_A$ for the orthogonal projection $L^2(X^m) \to L^2(X^A)$
and $P_A$ for the orthogonal projection
\begin{equation}%
  \label{eq:pa-proj}
    P_A : L^2(X^m) \to L^2(X^A) \cap \bigcap_{B \subsetneq A} L^2(X^B)^\perp.
\end{equation}
Here $L^2(X^B)^\perp$ is the space of functions $f(x_1, \dots, x_m)$ orthogonal to functions depending only on $(x_i : i \in B)$, i.e., such that $\EE_{x_i : i \notin B} f(x_1, \dots, x_m) = 0$ for any choice of $(x_i : i \in B)$.
Therefore the space on the right-hand side of~\eqref{eq:pa-proj} is the space of functions depending only on $(x_i : i \in A)$ and such that $\EE_{x_i} f(x_1, \dots, x_m) = 0$ for any $i \in A$.

The operators $P_A$, $Q_A$ are related via inclusion--exclusion rules:
\begin{align*}
    Q_A &= \sum_{B \subset A} P_B, \\
    P_A &= \sum_{B \subset A} (-1)^{|A \setminus B|} Q_B.
\end{align*}
Hence we have a kind of ``Fourier expansion''
\[
    f = \sum_{A \subset [m]} P_A f,
\]
for any function $f \in L^2(X^m)$ (which is only truly a Fourier expansion if $n = 2$ and $X^m$ is identified with $\FF_2^m$).
Applying this to $f = 1_S \in L^2(X^n)$,
\[
    1_S = \sum_{A \subset [n]} P_A 1_S.
\]
By the discussion above, $P_A 1_S$ can be thought of as the component of $1_S$ which depends exactly on $(x_i : i \in A)$ (and is orthogonal to all functions depending only on variables in a strict subset of $A$).
For example, $P_\emptyset 1_S$ is equal to the density $n!/n^n$.

The relevance of the $P_A$ projections is that any latin square tensor $\Lambda_{\sL}$ is diagonal with respect to this decomposition: that is,
\begin{equation}
    \label{eq:LambdaPA-sum}
    \Lambda(1_S, 1_S, 1_S) = \sum_{A \subset [n]} \Lambda(P_A1_S, P_A1_S, P_A1_S).
\end{equation}
This is a consequence of the following lemma.
\begin{lemma}
    Let $f \in L^2(X^n), g \in L^2(Y^n), h\in L^2(Z^n)$ and $A, B, C \subset [n]$.
    Then
    \[
        \Lambda(P_A f, P_B g, P_C h) = 0
    \]
    unless $A = B = C$.
\end{lemma}

\begin{proof}
    Assume it is not the case that $A = B = C$.
    By symmetry we may assume $A \not\subseteq B$, say $i \in A \setminus B$.
    We may also assume $P_A f = f$, $P_B g = g$, $P_C h = h$, by replacing $f, g, h$ with their images under $P_A, P_B, P_C$, respectively.
    Now consider
    \[
        \Lambda(f, g, h) = \EE_{(x,y,z) \in \sL^n} f(x) g(y) h(z).
    \]
    In particular consider the average over the variables $(x_i, y_i, z_i) \in \sL$. Since $i \notin B$, there is no dependence on $y_i$, so it is equivalent by the latin square property to average over all $(x_i, z_i) \in X \times Z$. Since $\EE_{x_i} f(x_1, \dots, x_m) = 0$, it follows that $\Lambda(f,g,h) = 0$.
\end{proof}

We now divide up the sum~\eqref{eq:LambdaPA-sum} according to the size $m$ of $A$.

\subsection{Major arcs}%
\label{subsec:major}

The terms in this decomposition where $A$ is very sparse (of size up to $c n^{1/2}$) form the \emph{major arcs}.

\begin{theorem}%
    \label{thm:major-arcs}
    There is a constant $c>0$ such that for $\log n < m \leq c n^{1/2}$,
    \[
        \sum_{\substack{A \subset [n] \\ |A| \le m }}
        \Lambda(P_A 1_S,\: P_A 1_S,\: P_A 1_S)
        = \pfrac{n!}{n^{n}}^3 e^{-1/2} \br{1 + O(m^2/n)}.
    \]
\end{theorem}

The proof is a mostly mechanical adaptation of~\cite[Section~4]{EMM}, which did not use group theory in an essential way.

\subsection{Sparse minor arcs}%
\label{subsec:sparse-minor}

The next range, the \emph{sparse minor arcs}, concerns $A$ of size up to $c n$ for some small absolute constant $c$.

\begin{theorem}%
    \label{thm:sparse-minor}
    There is a constant $c>0$ such that for $1 \leq m \leq cn$,
    \[
        \sum_{\substack{A \subset [n] \\ |A|=m}}
        \Lambda(|P_A1_S|, |P_A1_S|, |P_A1_S|)
        \leq \pfrac{n!}{n^n}^3 O(1)^{m + n/m} (m/n)^{m/8}.
    \]
\end{theorem}
Note $|\Lambda(P_A1_S, P_A1_S, P_A1_S)| \le \Lambda(|P_A1_S|, |P_A1_S|, |P_A1_S|)$ by the triangle inequality.
The point is we have an exponential-in-$m$ gain over the main term provided
\[
    m \log (n/m) > C'(m + n/m),
\]
for some large enough $C' > 0$. This would be satisfied as long as
\begin{equation}
    \label{eq:C-eps}
    C(n / \log n)^{1/2} < m < \eps n,
\end{equation}
for some large enough $C > 0$ and small enough $\eps > 0$.

We prove \Cref{thm:sparse-minor} by exhibiting a majorant for $|P_A 1_S|$ and then using generating function methods.

\subsection{Dense minor arcs}%
\label{subsec:high-entropy}

Finally we have the dense range, where $m \geq c n$.
Here we use quasirandomness.
To be precise we define a certain Markov chain on $X \times Y$,
with adjacency operator $\Adj$,
and we consider $\sL$ to be $\Adj$-quasirandom with parameter $\rho$ if $\Adj$ has a spectral gap at least $1 - \rho$, i.e., if the spectral radius of $\Adj - \cU$ is at most $\rho$, where $\cU$ is the projection to constants (the uniform distribution).
See \Cref{def:quasirandom}.

\begin{theorem}%
    \label{thm:dense-minor}
    For every $\eps>0$ there is $\rho>0$ such that if $\sL$ is $\Adj$-quasirandom with parameter $\rho$, then
    \[
        \sum_{\substack{A \subset [n] \\ |A| \geq \eps n}}
        |\Lambda(P_A1_S, P_A1_S, P_A1_S)|
        \leq \pfrac{n!}{n^n}^3 10^{-n}.
    \]
\end{theorem}


\subsection{Quasirandomness}%
\label{sub:outline-quasi}

It remains (for \Cref{thm:main-random,thm:quasirandom-groups-are-quasirandom}) to demonstrate that the latin squares in scope are quasirandom in this sense.
If $\sL$ is the multiplication table of a group $G$ we compute the entire spectrum of $\Adj$ and find $\rho = 1/D$ where $D$ is the minimal dimension of a nontrivial representation of $G$,
which shows that our notion of quasirandomness is equivalent to the usual one due to Gowers~\cite{gowers} in the case of groups.
For genuinely random latin squares we use recent work of Kwan, Sah, Sawhney, and Simkin~\cite{KSSS} to show that $\tr \Adj^6 = 1 + o(1)$ with high probability,
and this implies that $\rho = o(1)$.

\subsection{Proof of \Cref{thm:main-quasirandom}}%
\label{subsec:main-theorem}

Putting \Cref{thm:major-arcs,thm:sparse-minor,thm:dense-minor} together it is straightforward to deduce \Cref{thm:main-quasirandom}.

\begin{proof}[Proof of \Cref{thm:main-quasirandom}]%
  Let $C$ and $\eps$ be as in~\eqref{eq:C-eps} and $M := C(n / \log n)^{1/2}$. Theorems~\ref{thm:major-arcs},~\ref{thm:sparse-minor} and~\ref{thm:dense-minor} give us that for some $c>0$
  \begin{align*}
      \Lambda(1_S, 1_S, 1_S)
      & = \bigl(e^{-1/2} + O(M^2/n)\bigr) \, \pfrac{n!}{n^{n}}^3 \\
      & \qquad + \sum_{M < m \leq \eps n} O(e^{-c m}) \, \pfrac{n!}{n^n}^3 \\
      & \qquad + 10^{-n} \pfrac{n!}{n^n}^3,
  \end{align*}
  as long as $\sL$ is $\Adj$-quasirandom with parameter $\rho$ for small enough $\rho$ (depending on $\eps$).
  The choice of $M$ implies~\eqref{eq:main-Lambda-statement} and hence \Cref{thm:main-quasirandom}.
\end{proof}

\subsection{Layout of the paper}%
\label{sub:layout}

To prove \Cref{thm:major-arcs,thm:sparse-minor,thm:dense-minor} we need some background material on partition systems (\Cref{sec:partitions}) and on the primitive ``Fourier analysis'' of coordinate projections $Q_A$, $P_A$ discussed above (\Cref{sec:fourier}).
This builds on similar material from~\cite{EMM}.

Then, \Cref{sec:major-arcs,sec:true-sparse-minor,sec:dense-minor} give the proofs of the three key theorems above.
Finally, \Cref{sec:quasirandomness} proves the quasirandomness properties from \Cref{sub:outline-quasi}.

\section{Partitions and partition systems}%
\label{sec:partitions}

\subsection{Partitions}%
\label{subsec:partitions}

Most of our language relating to the partition lattice is standard.
\begin{enumerate}
    \item If $A$ is a set, $\Pi_A$ is the set of all partitions of $A$. If $A = [m]$ we will conserve brackets by writing simply $\Pi_m$.
    \item $\Pi_A^{(k)}$ is the set of partitions all of whose cells have size at most $k$.
    \item If $A \subset B$ then any partition of $A$ is identified with a partition of $B$ by adding singletons $\bigl\{\{b\} \colon b \in B \setminus A \bigr\}$. With this convention, $\Pi_A \subset \Pi_B$.
    \item The \emph{support} $\supp \pi$ of a partition $\pi \in \Pi_A$ is the union of the nonsingleton cells of $\pi$. It is the smallest set $B \subset A$ such that $\pi\in \Pi_B$.
    \item $\Pi'_A$ is the set of $\pi\in \Pi_A$ with $\supp \pi = A$.
    \item If $\pi, \pi' \in \Pi_A$, $\pi \leq \pi'$ means that $\pi$ is a \emph{refinement} of $\pi'$ (i.e., every cell of $\pi'$ is a union of cells of $\pi$).  Synonymously, $\pi'$ is a \emph{coarsening} of $\pi$.
    \item The \emph{meet} $\pi \wedge \pi'$ is the coarsest partition refining both $\pi$ and $\pi'$; the \emph{join} $\pi \vee \pi'$ is the finest partition coarsening both $\pi$ and $\pi'$.
    \item The partition $\{\{a\} : a \in A\}$ is the \emph{discrete partition}; the partition $\{A\}$ is the \emph{indiscrete partition}.
    \item The \emph{rank} of a partition $\pi \in \Pi_A$ is $\rank (\pi) = |A|-|\pi|$; equivalently it is the greatest $r$ such that there are partitions $\pi_0 < \pi_1 < \cdots < \pi_r = \pi$. (Note that $\rank (\pi)$ is meaningful without specifying $A$, unlike $|\pi|$; i.e., it is invariant under adding or removing singletons.)
    \item The \emph{M\"obius function} $\mu$ at $\pi \in \Pi_A$ is given by $\mu(\pi) = (-1)^{\rank (\pi)} \prod_{p \in \pi} (|p| - 1)!$. 
    \item A function $f : A \to X$ is \emph{$\pi$-measurable} if $f$ is constant on the cells of $\pi$.
      A subset $S \subset A$ is called $\pi$-measurable if $1_S$ is $\pi$-measurable.
    \item If $\pi \in \Pi_A$, $c_\pi \in L^2(X^A)$ is the indicator of $\pi$-measurability (i.e., $c_\pi(f)$ is $1$ if $f$ is $\pi$-measurable, and $0$ otherwise).
\end{enumerate}

The \emph{exponential formula} for partitions states
\begin{equation}
    \label{eq:exponential-formula}
    \sum_{m \geq 0} \frac{1}{m!} \sum_{\pi \in \Pi_m} \prod_{p \in \pi} x_{|p|}
    = \exp \br{\sum_{k \geq 1} \frac{1}{k!} x_k}.
\end{equation}
Here $x_1, x_2, \dots$ are formal variables.
We will apply~\eqref{eq:exponential-formula} several times in \Cref{sec:true-sparse-minor}.

\subsection{Partition systems}

In \Cref{sec:major-arcs,sec:true-sparse-minor} it will be essential to have good bounds on the quantity $\Lambda(c_{\pi_1},c_{\pi_2},c_{\pi_3})$ for $A \subset [n]$ and various choices $\pi_1,\pi_2,\pi_3 \in \Pi_A$.
This motivates the following definitions.
\begin{enumerate}
    \item A \emph{partition triple} on a set $A$ is a triple $\psystem = (\pi_1, \pi_2, \pi_3) \in \Pi_A^3$.
    \item We call $\psystem$ a \emph{partition system} if $\supp \pi_1 = \supp \pi_2 = \supp \pi_3$.
    \item The \emph{support} of $\psystem$ is $\supp\psystem = \supp \pi_1 \cup \supp \pi_2 \cup \supp \pi_3$.
\end{enumerate}

\begin{definition}[Combinatorial rank]%
    \label{def:crank}
    Let $\psystem = (\pi_1, \pi_2, \pi_3) \in \Pi_A^3$ be a partition triple.
    We write $S \subset \psystem$ to mean that $S \subset \pi_1 \sqcup \pi_2 \sqcup \pi_3$, i.e., $S$ is a collection of cells labelled 1, 2, or 3.
    A subset $S \subset \psystem$ is \emph{closed}
    (with respect to $\psystem$)
    if whenever $p_i \in \pi_i$ for $i = 1, 2, 3$ and $p_1 \cap p_2 \cap p_3 \neq \emptyset$, if two of $p_1, p_2, p_3$ are in $S$ then so is the third.
    The \emph{closure} $\langle S\rangle$ of $S$ is the intersection of all closed sets containing $S$.
    The \emph{combinatorial rank} of $\psystem = (\pi_1, \pi_2, \pi_3)$ is defined as
    \[
        \crank(\psystem) = 2 |A| - \min \left\{ |S| : S \subset \psystem, \langle S \rangle = \psystem \right\}.
    \]
\end{definition}

The motivation for combinatorial rank is the following bound.
\begin{lemma}%
    \label{lem:crank-bound}
    For a set $A$, partitions $\pi_1, \pi_2, \pi_3 \in \Pi_A$, and latin square $\sL \subset X \times Y \times Z$,
    \[
        0
        \le
        \Lambda(c_{\pi_1}, c_{\pi_2}, c_{\pi_3})
        \le
        n^{-\crank(\pi_1,\pi_2,\pi_3)}.
    \]
\end{lemma}
The idea of the proof is the same as for the related result~\cite[Lemma~4.6]{EMM}.
\begin{proof}
    The $\Lambda$ value is, by definition, $n^{-2|A|}$ times the number of triples of functions
    \[
        f_1 \colon A \to X, \qquad
        f_2 \colon A \to Y, \qquad
        f_3 \colon A \to Z
    \]
    such that $f_i$ is $\pi_i$-measurable for $i = 1, 2, 3$ and such that $(f_1(a), f_2(a), f_3(a)) \in \sL$ for all $a \in A$.
    Note we can think of $f_i$ as a function on the cells of $\pi_i$, since it is $\pi_i$-measurable.

    We claim that, given $S \subset \psystem$ with $\langle S \rangle = \psystem$, the triple $(f_1,f_2,f_3)$ is determined by the values of $f_i$ on cells in $S$.  Hence the number of such triples is at most $n^{|S|}$, giving the result.

    Indeed, suppose $f'_1, f'_2, f'_3$ is another triple of measurable functions with the same restriction to $S$.
    Let $W \subseteq \psystem$ be the set of all cells $p_i \in \pi_i$ such that $f_i|_{p_i} = f'_i|_{p_i}$.
    By hypothesis $W \supset S$.
    If $p_i \in \pi_i$ for $i=1,2,3$, $a \in p_1 \cap p_2 \cap p_3$, and two of $p_1,p_2,p_3$ are in $W$, then so is the third, as the triples $(f_1(a), f_2(a), f_3(a)), (f'_1(a), f'_2(a), f'_3(a)) \in \sL$ agree at two coordinates and so are equal by the latin square property.
    Hence $W$ is a closed set, so $W \supset \langle S \rangle = \psystem$ and $f_i=f'_i$, as required.
\end{proof}

This reduces the problem of bounding
$\Lambda(c_{\pi_1}, c_{\pi_2}, c_{\pi_3})$ from above to the problem of bounding $\crank(\pi_1,\pi_2,\pi_3)$ from below.
In~\cite{EMM} we did this using two slightly weaker notions of rank, called \emph{triple rank} and \emph{lower rank}, defined respectively as
\begin{align*}
    \trank(\psystem) &= \max_{\sigma \in S_3}\bigl(
    \rank(\pi_{\sigma(1)}) + \rank(\pi_{\sigma(2)} \vee \pi_{\sigma(3)})
    \bigr) \\
    \lrank(\psystem) &= \bigl(\rank(\pi_1) + \rank(\pi_2) + \rank(\pi_3) + \rank(\pi_1 \vee \pi_2 \vee \pi_3)\bigr)/2.
\end{align*}

\begin{lemma}%
    \label{lem:crank>=trank}
    $\crank(\psystem) \geq \trank(\psystem) \geq \lrank(\psystem)$.
\end{lemma}
\begin{proof}
    For the first inequality, let $S \subset \psystem$ contain all of $\pi_1$ and one cell of $\pi_2$ from each cell of $\pi_2 \vee \pi_3$. Then $|S| = |\pi_1| + |\pi_2 \vee \pi_3|$ and $\langle S \rangle = \psystem$, so $\crank(\psystem) \geq \rank(\pi_1) + \rank(\pi_2 \vee \pi_3)$, and equally for other permutations of 1, 2, 3.
    The second inequality was proved in~\cite[Lemma~4.8]{EMM}, and in any case will not be used in this paper.
\end{proof}

For continuity with~\cite{EMM}, we define the \emph{complexity} of a partition system $\psystem$ to be
\[
    \cx(\psystem) = \trank(\psystem) - |\supp \psystem|.
\]
The complexity of a partition system is nonnegative, and it is zero if and only if $\psystem = (\pi, \pi, \pi)$ for some \emph{matching} $\pi$, i.e., a partition of $A = \supp \pi$ into $|A|/2$ pairs.

\subsection{Combinatorial rank of matching systems}

In this subsection we compute $\crank(\pi_1, \pi_2, \pi_3)$ for all $(\pi_1, \pi_2, \pi_3) \in \Pi^{(2)}_A$, i.e., partition triples such that all cells of $\pi_1, \pi_2, \pi_3$ have size at most 2.
Where it applies, this is a significant improvement on what \Cref{lem:crank>=trank} gives us.

\begin{lemma}%
    \label{claim:cool-crank-value}
    Let $\pi_1, \pi_2, \pi_3 \in \Pi_A^{(2)}$.
    Suppose there are precisely $k$ cells $p \in \pi_1 \vee \pi_2 \vee \pi_3$
    such that $\pi_i |_p$ has full support (i.e., is a matching) for each $i \in [3]$.
    Then
    \[
        \crank(\pi_1,\pi_2,\pi_3) = \rank(\pi_1) + \rank(\pi_2) + \rank(\pi_3) - k.
    \]
\end{lemma}
\begin{proof}
    Since all terms are additive across cells of $\pi_1 \vee \pi_2 \vee \pi_3$,
    we may assume $\pi_1 \vee \pi_2 \vee \pi_3$ is indiscrete.
    In particular $k \in \{0, 1\}$, and $k = 0$ if and only if one of $\pi_1, \pi_2, \pi_3$ has a singleton.

    \emph{Case $k = 1$}:
    In this case $\pi_1, \pi_2, \pi_3$ are matchings, so
    \[
        \rank(\pi_1) = \rank(\pi_2) = \rank(\pi_3) = |A|/2,
    \]
    and we must show
    \[
        \crank(\pi_1, \pi_2, \pi_3) = 3|A|/2 - 1.
    \]
    Let $\cG$ be the multigraph whose vertex set is $A$ and with edges given by the cells of $\pi_1, \pi_2, \pi_3$ (which are all 2-cells).
    Clearly $\cG$ is $3$-regular, with $|A|$ vertices and $3|A|/2$ edges.
    Since $\pi_1 \vee \pi_2 \vee \pi_3$ is indiscrete, $\cG$ is connected.

    According to \Cref{def:crank}, we want to infect as few edges as possible in such a way that, if two infected edges incident at a vertex always spread infection to the third edge, then infection spreads to all edges.
    Note that for this to happen, it is necessary and sufficient to infect at least one edge in each cycle, since the edges that are uninfected at the end of the process form a subgraph with no vertex of degree $1$.
    Hence, equivalently, we want to delete as few edges as possible to get a forest.

    Since $\cG$ has $3|A|/2$ edges and any forest has at most $|A|-1$ edges, we must delete at least $|A|/2 + 1$ edges.
    Conversely, given any connected 3-regular multigraph, we can delete edges until we have a (simple) tree. Hence the minimal number of generators is precisely $|A|/2 + 1$, so $\crank(\pi_1, \pi_2, \pi_3) = 2|A| - (|A|/2+1) = 3|A|/2 - 1$, as claimed.

    \emph{Case $k = 0$}:
    In this case at least one of $\pi_1, \pi_2, \pi_3$ has a singleton, and we must show that
    \[
        \crank(\pi_1, \pi_2, \pi_3) = \rank(\pi_1) + \rank(\pi_2) + \rank(\pi_3).
    \]
    We define a graph $\cG$ as in the previous case but additionally for every singleton $\{v\} \in \pi_1 \sqcup \pi_2 \sqcup \pi_3$ we add an edge $\{v, *\}$, where $*$ is a special additional vertex at which infection does not spread.
    Since $\pi_1 \vee \pi_2 \vee \pi_3$ is indiscrete, $\cG \setminus *$ is connected.
    Since there is at least one singleton, $\cG$ is connected.
    Again we want to delete as few edges as possible to get a forest.
    The number of vertices in $\cG$ is $|A| + 1$ and the number of edges is $|\pi_1| + |\pi_2| + |\pi_3|$, so the number of edges we must delete is precisely
    $
        |\pi_1| + |\pi_2| + |\pi_3| - |A|.
    $
    Hence
    \[
        \crank(\pi_1, \pi_2, \pi_3) = 3|A| - |\pi_1| - |\pi_2| - |\pi_3| = \rank(\pi_1) + \rank(\pi_2) + \rank(\pi_3),
    \]
    as claimed.
\end{proof}

\section{The ``Fourier'' expansion of $1_S$}%
\label{sec:fourier}

Recall from \Cref{sec:outline} that $Q_A$ denotes the orthogonal projection $L^2(X^m) \to L^2(X^A)$ and $P_A$ denotes the orthogonal projection
\[
    P_A : L^2(X^m) \to L^2(X^A) \cap \bigcap_{B \subsetneq A} L^2(X^B)^\perp,
\]
and these operators are related via inclusion--exclusion rules:
\begin{align}
    Q_A &= \sum_{B \subset A} P_B, \nonumber\\
    P_A &= \sum_{B \subset A} (-1)^{|A \setminus B|} Q_B \label{eq:incl-excl}.
\end{align}
In this section we study the terms in the expansion
\[
    1_S = \sum_{A \subset [n]} P_A 1_S.
\]

To express some of the results it is convenient to use the linear map $U : \CC[z] \to \CC$ defined by
\[
    U(z^k) = \begin{cases}
    n^k / (n)_k &: k \leq n, \\
    0 &: k > n.
    \end{cases}
\]
Here $(n)_k = n(n-1) \cdots (n-k+1)$.

\subsection{Formulas for \texorpdfstring{$P_A1_S$}{P\_A1\_S}}

Let $S_A \subset X^A$ denote the set of injections $A \to X$.
Thus if $|A| = m$, $|S_A| = (n)_m$.

\begin{lemma}%
    \label{lem:Q_AS_B}
    If $A \subset [n]$ and $|A|=m$,
    \[
    Q_A 1_S
    = \frac{n!}{n^n} \frac{n^m}{(n)_m} 1_{S_A}
    .
    \]
\end{lemma}

\begin{proof}
    A function $f \colon A \to X$ can be extended to a bijection $[n] \to X$ in $(n - m)!$ ways if $f$ is injective and $0$ ways otherwise, and by definition $Q_A 1_S(f)$ is the number of such extensions normalized by $1/n^{n-m}$. 
\end{proof}

\begin{lemma}[{\cite[Lemma~4.3]{EMM}}]%
  \label{mobius-inversion-1}
    \[
        1_{S_A} = \sum_{\pi\in \Pi_A} \mu(\pi) c_\pi.
    \]
\end{lemma}

\begin{lemma}%
  \label{lem:P_A-on-1_S}
    If $A \subset [n]$ and $|A| = m$ then
    \[
        P_A 1_S = \frac{n!}{n^n} \frac{n^m}{(n)_m}
        \sum_{\pi\in \Pi'_A} \mu(\pi) P_A c_\pi.
    \]
\end{lemma}
\begin{proof}
    Combining the previous two lemmas,
    \[
        P_A 1_S
        = P_A Q_A 1_S
        = \frac{n!}{n^n} \frac{n^m}{(n)_m} \sum_{\pi \in \Pi_A} \mu(\pi) P_A c_\pi.
    \]
    Since $c_\pi \in L^2(X^{\supp \pi})$, only the terms with $\supp \pi = A$ survive.
\end{proof}

We can use $U$ to give another formula for $P_A1_S$.
If $x \in X^A$ (i.e., $x : A \to X$), the \emph{kernel} $\ker x \in \Pi_A$ of $x$ is the level set partition
\[
    \ker x = \bigl\{ x^{-1}(t) : t \in X, x^{-1}(t) \neq \emptyset\bigr\}.
\]
Note that
\[
    c_\pi(x) = 1 \iff x~\text{is $\pi$-measurable} \iff \pi \leq \ker x.
\]

\begin{lemma}%
    \label{prop:PA1S-sparseval-like-expression}
    Let $A \subset [n]$, $|A| = m$.
    For $x \in X^n$, let $\pi = \ker(x|_A) \in \Pi_A$. Then
    \[
        P_A1_S(x) = (-1)^{\rank(\pi)} \frac{n!}{n^n} U \prod_{p \in \pi} (|p| z - 1).
    \]
\end{lemma}
\begin{proof}
    From~\eqref{eq:incl-excl} and \Cref{lem:Q_AS_B} we have
    \[
        P_A 1_S = \frac{n!}{n^n} \sum_{B \subset A} (-1)^{|A \setminus B|} \frac{n^{|B|}}{(n)_{|B|}} 1_{S_B}.
    \]
    Now, the sets $B$ such that $x|_B$ is injective are precisely those which intersect
    each cell of $\pi$ in at most one point.
    Hence
    \begin{align*}
        P_A1_S(x)
        &= \frac{n!}{n^n} U \sum_{B \subset A} (-1)^{|A \setminus B|} z^{|B|} 1_{S_B}(x) \\
        &= \frac{n!}{n^n} (-1)^{|A| - |\pi|} U \prod_{p \in \pi} (|p| z - 1).\qedhere
    \end{align*}
\end{proof}

\subsection{Sparseval}%
\label{sec:sparseval}
The word \emph{sparseval} is our playful term for the computation of $\|P_A f\|_2^2$ for any $A \subset [n]$. This is possible by inclusion--exclusion and orthogonality: since
\[
    \|Q_A f\|_2^2 = \sum_{B \subset A} \|P_B f\|_2^2,
\]
it follows that
\begin{equation}
    \label{eq:abstract-sparseval}
    \|P_A f\|_2^2 = \sum_{B \subset A} (-1)^{|A\setminus B|} \|Q_B f\|_2^2.
\end{equation}

\begin{lemma}%
    \label{lem:L2-to-sparseval}
    If $A \subset [n]$ and $|A| = m$,
    \[
        \|P_A 1_S\|_2^2 = \pfrac{n!}{n^n}^2 U\bigl((z-1)^m\bigr).
    \]
\end{lemma}
\begin{proof}
    Note that $\|1_{S_B}\|_2^2 = (n)_{|B|}/n^{|B|}$ for every $B\subset A$. Hence, from~\eqref{eq:abstract-sparseval} and \Cref{lem:Q_AS_B},
    \[
        \|P_A 1_S\|_2^2 \pfrac{n!}{n^n}^{-2}
        = \sum_{B \subset A} (-1)^{|A \setminus B|} \frac{n^{|B|}}{(n)_{|B|}}
        = U\bigl((z-1)^m\bigr).
    \]
\end{proof}

\begin{proposition}%
    \label{prop:umbral-sparseval}
    Assume $0 \leq m \leq n$ and let $t = m/n$.
    Then
    \[
        0 \leq U\bigl((z-1)^m\bigr)
        \ll
        \binom{n}{m}^{-1} e^{s(t) n}
        ,
    \]
    where
    \[
        s(t)
        = t^{1/2} - t \log t^{1/2} - (1-t) \log(1 + t^{1/2})
        .
    \]
    In particular
    \[
        U\bigl((z-1)^m\bigr) \leq e^{O(m)} (m/n)^{m/2}.
    \]
\end{proposition}

\begin{proof}[Sketch]
    The inequality $U\bigl((z-1)^m\bigr) \geq 0$ follows from the previous lemma.
    For the main claim, by expanding we have
    \[
        U\bigl((z-1)^m\bigr) = \frac1{n!} \sum_{k=0}^m \binom{m}{k} (-1)^{m-k} n^k (n-k)!,
    \]
    and this can be identified as $\binom{n}{m}^{-1}$ times the coefficient of $X^m$ in
    $e^{nX} / (1+X)^{n-m+1}$.
    The stated bound follows by taking a contour integral (chosen in the spirit of the saddle-point method) to extract the coefficient.
    For details see~\cite[bound for the sum in (5.4)]{EMM-cyclic}.
    Extra care is needed for $t$ near 1, but we omit the details because we will not use the claim for $t > 1/2$. The second bound follows by Stirling's formula.
\end{proof}

The following corollary will not be used but is included for interest.

\begin{corollary}
    The sign of $P_A1_S(x)$ is $(-1)^{\rank(\ker x|_A)}$.
\end{corollary}
\begin{proof}
    Let $\pi = \ker(x|_A)$.
    By \Cref{prop:PA1S-sparseval-like-expression} it suffices to prove that
    \[
        U\prod_{p \in \pi} (|p|z - 1) > 0.
    \]
    There are nonnegative integers $r_\omega \geq 0$ such that
    \[
        \prod_{p \in \pi} (|p|z - 1) = \prod_{p \in \pi} \bigl(|p|(z - 1) + (|p| - 1)\bigr) = \sum_{\omega \subset \pi} r_\omega (z-1)^{|\omega|}.
    \]
    Hence the claim follows from $U\bigl((z-1)^m\bigr) \geq 0$.
\end{proof}

\section{Major arcs}%
\label{sec:major-arcs}

\def\frakS{\mathfrak{S}}

The goal in this section is to prove Theorem~\ref{thm:major-arcs}.
Define
\begin{align*}
    & \frakS_m = \sum_{2k \leq m} \frac{(-1)^k}{2^k k!}, \\
    & M_m = \sum_{\substack{A \subset [n] \\ |A| \leq m}} \Lambda(P_A1_S, P_A1_S, P_A1_S).
\end{align*}
Our aim is to prove that, for $m \leq c n^{1/2}$,
\begin{equation}
	\label{eq:true-major-arcs-goal}
    M_m = \bigl(\frakS_m + O(m^2/n)\bigr) \pfrac{n!}{n^n}^3.
\end{equation}
In particular this implies \Cref{thm:major-arcs}.

\subsection{The quantities $\gamma$ and $\gamma_0$}

From \Cref{lem:P_A-on-1_S} it is clear that to estimate $\Lambda(P_A1_S, P_A1_S, P_A1_S)$ it suffices to estimate $\Lambda(P_A c_{\pi_1}, P_A c_{\pi_2}, P_A c_{\pi_3})$ for every partition system $\psystem = (\pi_1, \pi_2, \pi_3)$ with support $A$ and aggregate the results with the appropriate weighting.
For continuity with~\cite[Section~4]{EMM}, we define the normalized quantities
\[
    \gamma_0(\psystem) = n^{\trank(\psystem)} \Lambda(c_{\pi_1}, c_{\pi_2}, c_{\pi_3})
\]
and
\[
    \gamma(\psystem)
    = n^{\trank(\psystem)} \Lambda(P_A c_{\pi_1}, P_A c_{\pi_2}, P_A c_{\pi_3})
\]
for any partition triple $\psystem = (\pi_1, \pi_2, \pi_3)$.
Note that
\[
    0 \leq \gamma_0(\psystem) \leq 1
\]
by \Cref{lem:crank-bound,lem:crank>=trank}.
Since $c_\pi \in L^2(X^{\supp \pi})$, $\gamma(\psystem) = 0$ unless $\psystem$ is a partition system.

\begin{lemma}%
    \label{lem:trivial-gamma-bound}
    Let $\psystem$ be a partition system with support $\supp \psystem = A$ of size $m$,
    and suppose $m'$ points of $A$ are contained in cells $\pi_1 \vee \pi_2 \vee \pi_3$ of size at least $3$.
    Then
    \[
        |\gamma(\psystem)| \leq 2^{m'}.
    \]
\end{lemma}
\begin{proof}[Sketch]
The idea is that
\begin{align*}
    \Lambda(P_A c_{\pi_1}, P_A c_{\pi_2}, P_A c_{\pi_3})
    &= \Lambda(c_{\pi_1}, c_{\pi_2}, P_A c_{\pi_3}) \\
    &= \sum_{B \subset A} (-1)^{|A \setminus B|}
    \Lambda(c_{\pi_1}, c_{\pi_2}, Q_B c_{\pi_3})
    \\
    &= \sum_{B \subset A} (-1)^{|A\setminus B|}
    \Lambda(Q_B c_{\pi_1}, Q_B c_{\pi_2}, Q_B c_{\pi_3}),
\end{align*}
and
\[
    Q_B c_\pi = n^{-\rank(\pi) + \rank(\pi|_B)} c_{\pi|_B},
\]
where
\[
    \pi|_B = \{p \cap B : p \in \pi, p \cap B \neq \emptyset\}.
\]
Let $\psystem|_B = (\pi_1 |_B, \pi_2|_B, \pi_3|_B)$.
Then, normalizing,
\[
    \gamma(\psystem)
    = \sum_{B \subset A} (-1)^{|A \setminus B|} \gamma_0(\psystem|_B) n^{-t(\psystem, B)},
\]
where
\[
    t(\psystem, B) = \trank(\psystem|_B) - \trank(\psystem) + \sum_{i=1}^3 \bigl(\rank(\pi_i) - \rank(\pi_i|_B)\bigr).
\]
In~\cite[Section~4]{EMM} we showed $t(\psystem, B) \geq 0$.
Since $\gamma_0(\psystem|_B) \in [0, 1]$ for all $B$ this shows $|\gamma(\psystem)| \leq 2^m$.
The stronger bound with $m'$ in place of $m$ follows by separating off the doubleton cells of $\pi_1 \vee \pi_2 \vee \pi_3$.
See~\cite[Section~4]{EMM} for details.
\end{proof}

\subsection{The \texorpdfstring{$M_m(z)$}{M\_m(z)} series}

For a partition triple $\psystem = (\pi_1, \pi_2, \pi_3)$ we use the shorthand
\[
  \mu(\psystem) = \mu(\pi_1)\, \mu(\pi_2)\, \mu(\pi_3).
\]
From Lemma~\ref{lem:P_A-on-1_S} we have
\[
  M_m
  = \pfrac{n!}{n^n}^3
  \sum_{|\supp \psystem| \leq m}
  \pfrac{n^{|\supp \psystem|}}{(n)_{|\supp \psystem|}}^3
  \mu(\psystem)
  \gamma(\psystem)
  n^{-\trank(\psystem)},
\]
where the sum is over all partition systems on $[n]$.
For $z\in\CC$ define
\[
  M_m(z) = \pfrac{n!}{n^n}^3
  \sum_{|\supp \psystem| \leq m} \pfrac{n^{|\supp \psystem|}}{(n)_{|\supp \psystem|}}^3 \mu(\psystem) \gamma(\psystem) n^{-|\supp\psystem|} z^{\cx(\psystem)}.
\]
As we have mentioned, $\cx(\psystem) \geq 0$ for any partition system $\psystem$, so $M_m(z)$ is a polynomial such that $M_m = M_m(1/n)$.
By bounding $M_m(z)$ and using some complex analysis we will show $M_m(1/n) \approx M_m(0)$, and then we will directly estimate $M_m(0)$.

\begin{proposition}%
  \label{prop:major-arcs-estimate}
  There is a constant $c>0$ such that, for $|z|^{1/2} \le c / m$, we have
  \[
    |M_m(z)| \ll \pfrac{n^m}{(n)_m}^2 \pfrac{n!}{n^n}^3.
  \]
\end{proposition}
\begin{proof}
By the definition of $M_m(z)$, the triangle inequality, and \Cref{lem:trivial-gamma-bound}, the quantity $|M_m(z)| /  \pfrac{n!}{n^n}^3$
is bounded by
\[
  \sum_{|A| \le m} n^{-|A|} \pfrac{n^{|A|}}{(n)_{|A|}}^3
  \sum_{\supp \psystem = A} 2^{m'(\psystem)} |\mu(\psystem)|\ |z|^{\cx (\psystem)},
\]
where $m'(\psystem)$ is the number of points of $\supp \psystem$ contained in cells of $\pi_1 \vee \pi_2 \vee \pi_3$ of size at least 3.
This exact sum was analyzed in~\cite[Section~4.4]{EMM}, and we showed that it is $O(n^m / (n)_m)^2$
provided $|z|^{1/2} < c/m$.
The proposition follows.
\end{proof}

\begin{corollary}%
  \label{cor:Mmf0-approx}
There is a constant $c>0$ such that, for $m < c n^{1/2}$,
\[
  |M_m - M_m(0)|
  \ll (m^2/n) \pfrac{n!}{n^n}^3.
\]
\end{corollary}
\begin{proof}
By the residue theorem,
\[
    M_m(u) - M_m(0)
    = \frac1{2\pi i} \oint_{|z|=R} \frac{M_m(z) u}{(z-u)z} \, dz
\]
as long as $|u| < R$.
Hence
\[
    \left|M_m(u) - M_m(0)\right|
    \leq \max_{|z|=R} |M_m(z)| \frac{|u|/R}{1-|u|/R}.
\]
Take $u = 1/n$ and $R = c^2/m^2$, where $c$ is as in the previous proposition. Then as long as $1/n < c^2/m^2$, i.e., $m < cn^{1/2}$, we get
\[
    |M_m - M_m(0)|
    \ll (1+n^{-1/2})^{2m} \pfrac{n^m}{(n)_m}^2 \pfrac{n!}{n^n}^3
    \frac{m^2/n}{1 - m^2/ (c^2 n)}.
\]
Hence as long as say $m < (c/2) n^{1/2}$ we get the claimed bound.
\end{proof}

\subsection{The constant term \texorpdfstring{$M_m(0)$}{M\_m(0)}}

By definition,
\[
  M_m(0)
  = \pfrac{n!}{n^n}^3
  \sum_{\substack{|\supp\psystem| \leq m \\ \cx (\psystem) = 0}}
  \pfrac{n^{|\supp\psystem|}}{(n)_{|\supp \psystem|}}^3
  \mu(\psystem)
  \gamma(\psystem)
  n^{-|\supp \psystem|}.
\]
As remarked, $\cx (\psystem) = 0$ if and only if $\psystem = (\pi, \pi, \pi)$ for some matching $\pi$.
In this case, if say $|\supp \psystem| = 2k$,
\begin{align*}
  &\frac{n^{|\supp \psystem|}}{(n)_{|\supp \psystem|}} = \frac{n^{2k}}{(n)_{2k}}
  = 1 + O(k^2/n),\\
  &\mu(\psystem) = \mu(\pi)^3 = (-1)^k,\\
  &\gamma(\psystem) = (1-1/n)^k.
\end{align*}
The last identity holds by a direct calculation analogous to~\cite[Lemma~4.10]{EMM}.
The number of matchings $\pi$ in $[n]$ of support size $2k$ is
\[
  \frac{(n)_{2k}}{2^k k!} = \frac{n^{2k}}{2^k k!} \bigl(1 + O(k^2/n)\bigr).
\]
Thus
\begin{align*}
  M_m(0)
  &=
  \pfrac{n!}{n^n}^3
  \sum_{k=0}^{\floor{m/2}} \frac{n^{2k}}{2^k k!} (-1)^k n^{-2k} \bigl(1 + O(k^2/n)\bigr)
  \\
  &= \pfrac{n!}{n^n}^3 \bigl(\frakS_m + O(1/n)\bigr)
  .
\end{align*}
By combining with Corollary~\ref{cor:Mmf0-approx} we have
\[
	M_m = \pfrac{n!}{n^n}^3 \bigl(\frakS_m + O(m^2/n)\bigr)
\]
provided $m < c n^{1/2}$.
This finishes the proof of~\eqref{eq:true-major-arcs-goal}.

\section{Sparse minor arcs}%
\label{sec:true-sparse-minor}

To prove \Cref{thm:sparse-minor} we need a bound on $\Lambda(|P_A 1_S|, |P_A 1_S|, |P_A 1_S|)$ for larger $|A|$.
Note that in any latin square $\sL' \subset (X',Y',Z')$,
\begin{align}%
  \nonumber
    |\Lambda(f,g,h)| &= \lvert \EE_{(x,y,z) \in \sL'} f(x) g(y) h(z) \rvert \\
                     &\le \EE_{x \in X'} |f(x)| \left\lvert \EE_{y,z \colon (x,y,z) \in \sL'} g(y) h(z) \right\rvert \le \|f\|_1 \|g\|_2 \|h\|_2   \label{eq:l1-l2-l2}
\end{align}
using the latin square property and Cauchy--Schwarz, and similarly permuting $f,g,h$.
One approach to \Cref{thm:sparse-minor} might be to find upper bounds on $|P_A 1_S(x)|$, pointwise or in $L^1$, and simply apply~\eqref{eq:l1-l2-l2}.
However, by itself this approach is too crude, even assuming optimal upper bounds.

Another idea is to seek a majorant for $|P_A 1_S|$ of the form
\begin{equation}%
  \label{eq:majorant-form}
  |P_A 1_S| \le \sum_{\pi \in \Pi_A} t_{\pi} c_\pi
\end{equation}
for some coefficients $t_{\pi} \ge 0$.
Then
\[
  \Lambda(|P_A 1_S|, |P_A 1_S|, |P_A 1_S|) \le \sum_{\pi_1,\pi_2,\pi_3 \in \Pi_A} t_{\pi_1} t_{\pi_2} t_{\pi_3} \Lambda(c_{\pi_1},c_{\pi_2},c_{\pi_3})
\]
and \Cref{lem:crank-bound}, together with generating function techniques, gives a way to control the right-hand side.
This bound is particularly effective if $\pi_i \in \Pi_A^{(2)}$, given \Cref{claim:cool-crank-value}.

Again this approach does not succeed by itself.
Our final argument works by decomposing $|P_A 1_S|$ into two pieces and combining the two techniques discussed above.

\subsection{A majorant for $|P_A 1_S|$}

\newcommand\sig[1]{\sigma_{#1}^{(\delta)}}

Throughout this section let $C>0$ be some large enough constant, $A \subset [n]$ and $|A| = m \leq n/C$.
Additionally, we let
\[
    \delta := (Cm/n)^{1/2}.
\]
Although we will always have this specific value of $\delta$ in mind, most of the results in this section only rely on $\delta \leq 1$.
The next proposition gives a useful bound for $|P_A 1_S|$.
For $\delta := (Cm/n)^{1/2}$, $r \geq 1$, and $\pi$ a partition define
\begin{align*}
    \sig{r} &=
    \begin{cases}
    \delta &: r = 1, \\
    r - 1 &: r > 1,
    \end{cases} \\
    \sig{\pi} &= \prod_{p \in \pi} \sig{|p|}.
\end{align*}

\begin{proposition}%
    \label{prop:PA1S-physical-bound}
    We have
    \[
        |P_A1_S(x)|
        \leq \frac{n!}{n^n} e^{\delta m}
        \sig{\ker x}
        \qquad
        (x \in X^A).
    \]
\end{proposition}
\begin{proof}
    From \Cref{prop:PA1S-sparseval-like-expression},
    \[
        P_A1_S(x) = (-1)^{\rank(\pi)} \frac{n!}{n^n} U\phi,
    \]
    where $\pi = \ker x$ and
    \begin{align*}
        \phi &= \prod_{p \in \pi} (|p| z - 1) = \sum_{\omega \subset \pi} r_\omega (z-1)^{|\omega|},\\
        r_\omega &= \prod_{p \in \pi \setminus \omega} (|p|-1) \prod_{p \in \omega} |p|.
    \end{align*}
    From \Cref{prop:umbral-sparseval} and crude estimates (Stirling's formula), for $0 \leq d \leq m$,
    \[
        U\bigl((z-1)^d\bigr)
        \leq (Cd/n)^{d/2}
        \leq (Cm/n)^{d/2}
        = \delta^{d}
    \]
    provided $C$ is large enough.
    Then
    \begin{align*}
        U\phi
        &\le
        \sum_{\omega \subset \pi}
        r_\omega \delta^{|\omega|} \\
        &= \prod_{p \in \pi} (|p| - 1 + |p| \delta) \\
        &= \sig{\pi} \prod_{p \in \pi : |p| > 1} \br{1 + \frac{|p|}{|p|-1} \delta} \\
        &\leq \sig{\pi} \br{1 + 2 \delta}^{m/2} \\
        &\leq \sig{\pi} e^{\delta m}
    \end{align*}
    as required.
\end{proof}

In light of the proposition, to find majorants for $|P_A 1_S|$ of the form~\eqref{eq:majorant-form} it suffices to find analogous bounds for $\sig\pi$.
Recall that $\Pi_A^{(k)}$ is the set of all $\pi \in \Pi_A$ having no part of size greater than $k$.
Let $r_k(\pi)$ be the number of $k$-cells in $\pi$
and let $r_{3+}(\pi) = \sum_{k \geq 3} r_k(\pi)$.

\begin{lemma}%
    \label{lem:partition-breaking}
    Let $\pi$ be a partition. 
    \begin{enumerate}
        \item
        \[
            \sig\pi \leq \sum \Bigl\{\sig{\pi'} : \pi' \leq \pi, \pi' \in \Pi^{(3)}
            \Bigr\}.
        \]
        \item
        \[
            \sig\pi \leq \sum\Bigl\{ \sig{\pi'} :
            \pi' \leq \pi,
            \pi' \in \Pi^{(4)},
            r_{3+}(\pi') = r_{3+}(\pi)
            \Bigr\}.
        \]
        \item
        \[
            \sig\pi \leq \delta^{-r_{3+}(\pi)} \sum\Bigl\{
            \sig{\pi'} :
            \pi' \leq \pi,
            \pi' \in \Pi^{(2)}
            \Bigr\}.
        \]
    \end{enumerate}
\end{lemma}
\begin{proof}
    Consider the first inequality.
    Both sides are multiplicative across cells of $\pi$, so we may assume $\pi$ is a single cell, say of size $r$.
    The inequality is trivial for $r \leq 3$ (since $\sig\pi$ is one of the summands on the right-hand side), so we may assume $r \geq 4$.
    Then it suffices to check
    \[
        r-1 \leq \sum_{\substack{2a + 3b = r}}
        \frac{r!}{2!^a a! 3!^b b!} 2^b.
    \]
    This is a calculation for $r \leq 10$ (say) and an uninteresting exercise for $r > 10$.

    Now consider the second inequality.
    This time, the right-hand side is not itself multiplicative over cells of $\pi$, but if we replace the condition $r_{3+}(\pi') = r_{3+}(\pi)$ by the stronger one
    \[
      \forall p \in \pi,\ |p| \ge 3:\ \text{there is exactly one } p' \in \pi' \text{ with } p' \subseteq p \text{ and } |p'| \ge 3
    \]
    then it becomes so, and it suffices to prove the corresponding stronger inequality.
    Now we may again assume that $\pi$ is an $r$-cell, and we may assume $r \geq 5$.
    Then we must check
    \[
        r-1 \leq \sum_{\substack{2a + 3b + 4c = r \\ b + c = 1}}
        \frac{r!}{2!^a a! 3!^b b! 4!^c c!} 2^b 3^c.
    \]
    Again we omit further details.

    Now consider the third inequality.
    Again it suffices to consider the case of an $r$-cell,
    and we may assume $r \geq 3$.
    Then the assertion is
    \[
        r-1 \leq \delta^{-1} \sum_{a + 2b = r} \frac{r!}{a! 2!^b b!} \delta^a.
    \]
    Since $\delta \leq 1$, it suffices to check
    \[
        r-1 \leq \sum_{\substack{a + 2b = r \\ a \leq 1}} \frac{r!}{a! 2!^b b!},
    \]
    which is again essentially a calculation.
\end{proof}

\begin{lemma}%
    \label{lem:sieve-bound}
    Let $x \in X^A$. Then
    \[
        |P_A1_{S_A}(x)|
        \leq
        \frac{n!}{n^n} e^{\delta m} \sum_{\pi \in \Pi_A^{(3)}} \sig\pi c_\pi(x).
    \]
\end{lemma}

\begin{proof}
    By \Cref{prop:PA1S-physical-bound},
    \[
        |P_A1_{S_A}(x)|
        \leq \frac{n!}{n^n} e^{\delta m} \sig{\ker x}.
    \]
    By Lemma~\ref{lem:partition-breaking}(1),
    \[
        \sig{\ker x} \leq \sum\Bigl\{\sig\pi : \pi \leq \ker x, \pi \in \Pi_A^{(3)}\Bigr\}
        = \sum_{\pi \in \Pi_A^{(3)}} \sig\pi c_\pi(x).\qedhere
    \]
\end{proof}
%

\subsection{A splitting of $|P_A 1_S|$}

We can use the bound on $|P_A1_S|$ given in the previous section to bound the $L^1$ norm of $P_A1_S$, but the bound would not be strong enough for what we need.
To go further, we break up $R := |P_A 1_S|$ into two parts, a part whose $L^1$ norm we can control better, and a part we can analyze separately.
Fix $\eps\geq 0$ and let
\[
    \Pi^\sharp = \{ \pi \in \Pi_A : r_{3+}(\pi) < \eps m \}.
\]
Let $\Pi^\flat = \Pi_A \setminus \Pi^\sharp$.
Define
\begin{align*}
    R^{\sharp}(x) &= 1_{\Pi^\sharp}(\ker x) R(x), \\
    R^{\flat}(x) &= 1_{\Pi^\flat}(\ker x) R(x).
\end{align*}
Clearly $R = R^\sharp + R^\flat$.

\begin{lemma}%
    \label{lem:R-pointwise-bounds}
    We have
    \begin{align*}
    R^\flat
    &\leq \frac{n!}{n^n} e^{\delta m}\sum_{\pi \in \Pi^\flat \cap \Pi^{(4)}} \sig\pi c_\pi
    ,\\
    R^\sharp
    &\leq \frac{n!}{n^n} e^{\delta m} \sum_{\pi \in \Pi^\sharp \cap \Pi^{(4)}} \sig\pi c_\pi
    ,
    \\
    R^\sharp
    &\leq \frac{n!}{n^n} e^{\delta m} \delta^{-\eps m} \sum_{\pi \in \Pi^{(2)}} \sig\pi c_\pi
    .
    \end{align*}
\end{lemma}
\begin{proof}
    By \Cref{prop:PA1S-physical-bound},
    \[
        R(x) \le \frac{n!}{n^n} e^{\delta m} \sig{\ker x}.
    \]
    Suppose $\ker x \in \Pi^\flat$.
    Then by Lemma~\ref{lem:partition-breaking}(2),
    \[
        \sigma_{\ker x} \leq \sum\{
            \sigma_\pi : \pi \leq \ker x, \pi \in \Pi^{(4)}, r_{3+}(\pi) \geq \eps m
        \}.
    \]
    This proves the bound on $R^\flat$.
    The first bound on $R^\sharp$ is proved identically.
    The second is proved in the same way using instead Lemma~\ref{lem:partition-breaking}(3).
\end{proof}

\begin{corollary}%
    \label{cor:rflat-l1-bound}
    We have
    \[
        \| R^\flat \|_1
        \ll \frac{n!}{n^n} e^{O(m)} (m/n)^{(1 + \eps)m/2}.
    \]
\end{corollary}
\begin{proof}
    Using the previous lemma, $\delta \le 1$ and $\|c_\pi\|_1 = n^{-\rank(\pi)}$,
    \[
        \|R^\flat\|_1 \leq \frac{n!}{n^n} e^{m} \sum_{\pi \in \Pi^\flat \cap \Pi^{(4)}} \sig\pi n^{-\rank(\pi)}.
    \]
    Let
    \[
        \alpha_r(x, w) = \sum_{\pi \in \Pi_r^{(4)}} \sig\pi x^{\rank(\pi)} w^{r_{3+}(\pi)}.
    \]
    Then, for real $w\geq 1$,
    \[
        \sum_{\pi \in \Pi^\flat \cap \Pi^{(4)}} \sig\pi n^{-\rank(\pi)} \leq w^{-\eps m} \alpha_m(1/n, w).
    \]
    Using the exponential formula~\eqref{eq:exponential-formula} with $x_k = \sig k x^{k-1} y^k$ for $k=1,2$, $x_k = \sig k w x^{k-1} y^k$ for $k=3,4$, and $x_k=0$ for $k\geq 5$, we obtain
    \[
        \sum_{r\geq 0} \frac1{r!} \alpha_r(x, w) y^r
    = \exp(\delta y + xy^2/2 + w x^2 y^3 / 3 + w x^3 y^4 / 8).
    \]
    Hence for real $y>0$,
    \[
        w^{-\eps m} \alpha_m(x, w)
        \leq
        \frac{m!}{w^{\eps m} y^m} \exp(\delta y + xy^2/2 + wx^2 y^3/3 + wx^3 y^4/8).
    \]
    Putting $x = 1/n$, $y = (mn)^{1/2}$, and $w = (n/m)^{1/2}$, we get
    \[
        w^{-\eps m} \alpha_m(1/n, w)
        \leq
        \frac{m!}{(n/m)^{\eps m / 2}(mn)^{m/2}} e^{O(m)}.
    \]
    This proves what we want.
\end{proof}

\begin{corollary}%
    \label{cor:R-telescope}
    We have
    \[
        \Lambda(R,R,R) \le \Lambda(R^\sharp, R^\sharp, R^\sharp)
        + \pfrac{n!}{n^n}^3
        O(1)^m (m/n)^{(1 + \eps/2) m}
        .
    \]
\end{corollary}
\begin{proof}
    We have $\|R^\sharp\|_2 \leq \|R\|_2$ since $0 \leq R^\sharp \leq R$ pointwise.
    Hence, from~\eqref{eq:l1-l2-l2},
    \begin{align*}
    \Lambda(R,R,R) &= \Lambda(R^\flat, R, R) + \Lambda(R^\sharp, R^\flat, R) + \Lambda(R^\sharp, R^\sharp, R^\flat) + \Lambda(R^\sharp, R^\sharp, R^\sharp) \\
    &\le \Lambda(R^\sharp, R^\sharp, R^\sharp) + 3 \|R\|_2^2 \|R^\flat\|_1.
    \end{align*}
    From sparseval (\Cref{lem:L2-to-sparseval} and \Cref{prop:umbral-sparseval}),
    \[
        \|R\|_2^2 \ll \pfrac{n!}{n^n}^2 e^{O(m)} (m/n)^{m/2}.
    \]
    Combining with \Cref{cor:rflat-l1-bound} gives the bound.
\end{proof}

\subsection{The contribution from $R^\sharp$}

Finally we must bound $\Lambda(R^\sharp, R^\sharp, R^\sharp)$.
From Lemma~\ref{lem:R-pointwise-bounds},
\begin{equation}
    \label{eq:R_sharp<Q}
    R^\sharp \leq \frac{n!}{n^n} e^{\delta m} \delta^{-\eps m} Q
    \leq \frac{n!}{n^n} e^{O(m)} (m/n)^{-\eps m/2} Q,
\end{equation}
where
\[
    Q = \sum_{\pi \in \Pi^{(2)}} \sig\pi c_\pi.
\]
Hence it suffices to bound $\Lambda(Q,Q,Q)$.
The key ingredient for this is the knowledge of the exact value of combinatorial rank for $\pi_1,\pi_2,\pi_3 \in \Pi^{(2)}$ (\Cref{claim:cool-crank-value}).

\def\matchings{\mathcal{M}}

\begin{lemma}%
    \label{lem:LambdaQQQ}
    \[
        \Lambda(Q, Q, Q)
        \le (m/n)^{3m/2} e^{O(m + n/m)}.
    \]
\end{lemma}
\begin{proof}
    Let $\matchings_A \subset \Pi_A^{(2)}$ be the set of matchings (partitions all of whose cells have size 2).
    For $\pi_1,\pi_2,\pi_3 \in \Pi^{(2)}_A$,
    let $k(\pi_1, \pi_2, \pi_2)$ be the number of cells $p \in \pi_1 \vee \pi_2 \vee \pi_3$ such that $\pi_i|_p \in \matchings_p$ for each $i \in [3]$.
    Then, from \Cref{lem:crank-bound} and \Cref{claim:cool-crank-value},
    \[
        \Lambda(c_{\pi_1}, c_{\pi_2}, c_{\pi_3})
        \leq n^{k(\pi_1, \pi_2, \pi_3) - \rank(\pi_1)-\rank(\pi_2)-\rank(\pi_3)}.
    \]
    Hence
    \begin{align*}
        \Lambda(Q,Q,Q)
        &=
        \sum_{\pi_1, \pi_2, \pi_3 \in \Pi^{(2)}_A}
        \sig{\pi_1} \sig{\pi_2} \sig{\pi_3}
        \Lambda(c_{\pi_1}, c_{\pi_2}, c_{\pi_3})
        \\&\leq
        \sum_{\pi_1, \pi_2, \pi_3 \in \Pi^{(2)}_A}
        \sig{\pi_1} \sig{\pi_2} \sig{\pi_3} n^{k(\pi_1, \pi_2, \pi_3) - \rank(\pi_1) - \rank(\pi_2) - \rank(\pi_3)}
        \\&=
        \sum_{\pi \in \Pi_A}
        \prod_{p \in \pi}
        \sum_{\substack{
                \pi_1, \pi_2, \pi_3 \in \Pi^{(2)}_p \\
                \pi_1 \vee \pi_2 \vee \pi_3 = \{p\}
        }}
        n^{k(\pi_1, \pi_2, \pi_3)}
        \prod_{i \in [3]} \sig{\pi_i} n^{-{\rank(\pi_i)}}
        .
    \end{align*}
    In the last sum above, since $\pi_1 \vee \pi_2 \vee \pi_3 = \{p\}$, $k(\pi_1, \pi_2, \pi_3)$ is $0$ or $1$ according to whether $\pi_1, \pi_2, \pi_3 \in \matchings_p$.
    Splitting the sum according to these cases,
    \begin{align*}
        \Lambda(Q,Q,Q)
        &\leq
        \sum_{\pi \in \Pi_A}
        \prod_{p \in \pi}
        \bigg(
        \sum_{\substack{
                \pi_1, \pi_2, \pi_3 \in \Pi^{(2)}_p \\
                \pi_1 \vee \pi_2 \vee \pi_3 = \{p\}
        }}
        \prod_{i \in [3]} \sig{\pi_i} n^{-{\rank(\pi_i)}}
        \\&\qquad+
        \sum_{\substack{
                \pi_1, \pi_2, \pi_3 \in \matchings_p \\
                \pi_1 \vee \pi_2 \vee \pi_3 = \{p\}
        }}
        n
        \prod_{i \in [3]} \sig{\pi_i} n^{-{\rank(\pi_i)}}
        \bigg)
        .
    \end{align*}
    In the second sum we will ignore the constraint $\pi_1\vee \pi_2 \vee \pi_3 = \{p\}$;
    in the first sum we will use only $\rank(\pi_1) + \rank(\pi_2) + \rank(\pi_3) \geq \rank( \pi_1\vee \pi_2 \vee \pi_3 ) = |p|-1$.

    Fix parameters $w_r \geq 1$ for all $r \geq 1$.
    Define
    \begin{align*}
        \alpha_r(x) &= \sum_{\pi \in \Pi_r^{(2)}} \sig\pi x^{\rank(\pi)}
        ,
        \\
        \alpha_r'(x) &= \sum_{\pi \in \matchings_r} \sig\pi
        x^{\rank(\pi)}
        = |\matchings_r| x^{r/2}
        ,
        \\
        \beta_r(x) &= \sum_{\pi \in \Pi_r} \prod_{p \in \pi}
        \br{
            w_{|p|}^{-(|p|-1)} \alpha_{|p|}(w_{|p|}x)^3
            + x^{-1} \alpha'_{|p|}(x)^3
        }
        .
    \end{align*}
    Then, by the discussion above,
    \[
        \Lambda(Q,Q,Q) \leq \beta_m(1/n).
    \]
    Three applications of the exponential formula~\eqref{eq:exponential-formula}
    give
    \begin{align}
        \sum_{r \geq 0} \frac{y^r}{r!} \alpha_r(x)
        &= \exp(\delta y + x y^2 / 2),
        \label{eq:ef-alpha}
        \\
        \sum_{r \geq 0} \frac{y^r}{r!} \alpha_r'(x)
        &= \exp(x y^2 / 2),
        \label{eq:ef-alpha'}
        \\
        \sum_{r \geq 0} \frac{y^r}{r!} \beta_r(x)
        &= \exp \br{
            \sum_{r \geq 1} \frac{y^r w_r^{-r+1}\alpha_r(w_r x)^3}{r!}
            + \sum_{r\geq 2~\text{even}} \frac{y^r x^{-1} \alpha_r'(x)^3}{r!}
        }
        \label{eq:ef-beta}
    .
    \end{align}
    From~\eqref{eq:ef-alpha}, for real $y>0$,
    \[
        \alpha_r(x) \leq \frac{r!}{y^r}
        \exp(\delta y + xy^2/2).
    \]
    Replacing $x$ with $w_r x$, putting $w_r = \delta^2 / (xr)$ (we will ensure later that $w_r \geq 1$ for $1 \leq r \leq m$) and $y = r / \delta$ gives
    \[
        w_r^{-r+1} \alpha_r(w_r x)^3 \leq
        e^{O(r)}
        r^{r}
        \delta^{r+2}
        x^{r-1}
        .
    \]
    From~\eqref{eq:ef-alpha'} with $y = (r/x)^{1/2}$ we have
    \[
        \alpha'_r(x) \leq \frac{r!}{y^r} \exp(xy^2/2) \asymp r^{1/2} (rx/e)^{r/2}
    \]
    (alternatively, this follows directly from $\alpha'_r(x) = |\matchings_r| x^{r/2}$).
    Hence, from~\eqref{eq:ef-beta} for $x, y > 0$,
    \begin{equation}
        \label{eq:beta_m-bound}
        \beta_m(x) \leq \frac{m!}{y^m} \exp b(x, y),
    \end{equation}
    where $b$ is the truncated sum
    \begin{align*}
        b(x, y)
        &=
            \sum_{r=1}^m \frac{y^r w_r^{-r+1} \alpha_r(w_r x)^3}{r!}
            + \sum_{r=2}^m \frac{y^r x^{-1} \alpha_r'(x)^3}{r!}
        \\
        &\ll
            \sum_{r=1}^m e^{O(r)} \delta^{r+2} x^{r-1} y^r
            + \sum_{r=2}^m r^{O(1)} (e^{-1/2} r^{1/2}x^{3/2} y)^r x^{-1}
        .
    \end{align*}
    Inserting $x = 1/n$ and $\delta = (Cm/n)^{1/2}$,
    \[
        b(1/n, y)
        \ll
        \sum_{r=1}^m O(m^{1/2} y /n^{3/2} )^r m
        + \sum_{r=2}^m r^{O(1)} (e^{-1/2} r^{1/2}y / n^{3/2})^r n
        .
    \]
    Note that $w_r = Cm/r$, and this is indeed at least 1 for $r \leq m$ since we may assume $C \geq 1$.
    Finally, let $y = c n^{3/2} / m^{1/2}$ for a sufficiently small constant $c > 0$.
    Then
    \[
        b(1/n, y) \ll m + n/m.
    \]
    Hence, from~\eqref{eq:beta_m-bound},
    \[
        \Lambda(Q,Q,Q) \leq \beta_m(1/n) \leq \frac{m!}{y^m} \exp b(1/n, y) \ll (m/n)^{3m/2} e^{O(m + n/m)},
    \]
    as claimed.
\end{proof}

Putting the last few results together, we have the following theorem, which clearly implies \Cref{thm:sparse-minor}.

\begin{theorem}
    We have
    \[
        \Lambda(|P_A1_S|, |P_A1_S|, |P_A1_S||)
        \leq \pfrac{n!}{n^n}^3 (m/n)^{9m/8} e^{O(m + n/m)}.
    \]
\end{theorem}
\begin{proof}
    From Corollary~\ref{cor:R-telescope},
    \[
        \Lambda(R, R, R)
        \leq \Lambda(R^\sharp, R^\sharp, R^\sharp)
        + \pfrac{n!}{n^n}^3 e^{O(m)} (m/n)^{(1 + \eps/2)m}.
    \]
    By~\eqref{eq:R_sharp<Q} and the previous lemma the main term is
    \begin{align*}
        \Lambda(R^\sharp, R^\sharp, R^\sharp)
        &\leq \pfrac{n!}{n^n}^3 e^{O(m)} (m/n)^{-3\eps m/2} \Lambda(Q, Q, Q) \\
        &\leq \pfrac{n!}{n^n}^3 (m/n)^{(1 - \eps)3m/2} e^{O(m + n/m)}.
    \end{align*}
    Set $\eps = 1/4$.
\end{proof}
%

\section{Dense minor arcs}%
\label{sec:dense-minor}

\begin{figure}
    \centering
    \begin{tikzpicture}[scale=0.5]
      \def\er{0.5}
      \def\ex{4}
      \def\ey{3}

      \draw[fill=gray!20] (-\ex,\ey) -- (\ex,\ey) -- (0,0) -- cycle;
      \draw[fill=gray!20] (-\ex,-\ey) -- (\ex,-\ey) -- (0,0) -- cycle;

      \draw[fill=red!20] (-\ex,\ey) node{$x$} circle (\er);
      \draw[fill=green!20] (\ex,\ey) node{$y'$} circle (\er);

      \draw[fill=green!20] (-\ex,-\ey) node{$y$} circle (\er);
      \draw[fill=red!20] (\ex,-\ey) node{$x'$} circle (\er);

      \draw[fill=blue!20] (0,0) node{$z$} circle (\er);
  	\end{tikzpicture}
  	\caption{A transition $(x, y) \mapsto (x', y')$ in the Markov chain}%
    \label{fig:markov-chain}
\end{figure}
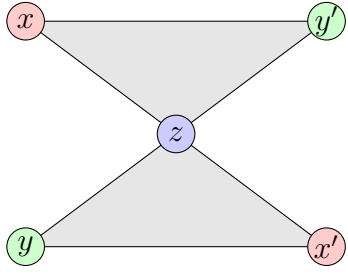

Define a Markov chain on $X \times Y$ as follows.
If the current state is $(x, y)$, pick uniformly at random $z \in Z$.
The next state is $(x', y')$, where $x'$ and $y'$ are the unique solutions to
\[
    (x, y', z), (x', y, z) \in \sL
\]
(see \Cref{fig:markov-chain}).
Let $\Adj$ be the transition operator for this Markov chain:
\[
    \Adj(f)(x, y) = \frac1n \sum_{(x, y', z), (x', y, z) \in \sL} f(x', y').
\]
The Markov chain is reversible with uniform stationary distribution, so $\Adj$ is self-adjoint and has the constant function on $X \times Y$ as a $1$-eigenvector.
Let $\cU$ be the projection to constants:
\[
    \cU(f)(x, y) = \frac1{n^2} \sum_{x', y'} f(x', y').
\]
\begin{definition}%
  \label{def:quasirandom}
  We say $\sL$ is \emph{$\Adj$-quasirandom} with parameter $\rho$ if $\Adj - \cU$ has spectral radius at most $\rho$.
\end{definition}
In particular, $\rho < 1$ if and only if the Markov chain is connected, and in general $\rho$ measures the rate of mixing.

\begin{remark}%
  \label{rem:l20}
  For a finite set $T$, let $L^2(T)_0$ denote the subspace $\bigl\{ f \in L^2(T) \colon \EE f = 0 \bigr\}$.
  Then equivalently, $\sL$ is $\Adj$-quasirandom with parameter $\rho$ if the restriction $\Adj|_{L^2(X \times Y)_0}$ has spectral radius at most $\rho$.
\end{remark}

All our applications of quasirandomness go through the following lemma.
\begin{lemma}%
  \label{lem:quasi-use}
  Assume $\sL$ is $\Adj$-quasirandom with parameter $\rho$ and\footnote{Note that the $m=1$ case of~\eqref{eq:quasi} does not obviously imply the general case: the operator-type norm for trilinear forms does not behave well under taking tensor powers.} let $m \ge 1$.
Then
\begin{equation}%
  \label{eq:quasi}
    |\Lambda(f, g, h)| \leq \rho^{m/2} \|f\|_2 \|g\|_2 \|h\|_2
\end{equation}
for all $f \in L^2(X)_0^{\otimes m}$, $g \in L^2(Y)_0^{\otimes m}$, $h \in L^2(Z)_0^{\otimes m}$.
\end{lemma}

\begin{remark}%
  \label{rem:0-tensor}
  Identifying $L^2(X)^{\otimes m}$ with $L^2(X^m)$ in the usual way, $L^2(X)_0^{\otimes m}$ is identified with the subspace $\im P_{[m]} \subseteq L^2(X^m)$; see~\eqref{eq:pa-proj}.
\end{remark}

\begin{proof}[Proof of \Cref{lem:quasi-use}]
    By Cauchy--Schwarz,
    \begin{align*}
        |\Lambda(f, g, h)|
        &= \left| \EE_{(x, y, z) \in \sL^m} f(x) g(y) h(z) \right| \\
        &\leq \br{ \EE_z \left | \EE_{x, y: (x, y, z) \in \sL^m} f(x) g(y)\right|^2}^{1/2} \|h\|_2 \\
        &= \br{\EE_{z, x, y, x', y' : (x, y, z), (x', y', z) \in \sL^m} f(x) g(y) {\bar f}(x') {\bar g}(y')}^{1/2} \|h\|_2 \\
        &= \langle \Adj^{\otimes m}(f \otimes \bar g), f \otimes \bar g \rangle^{1/2} \|h\|_2.
    \end{align*}
    Note $\|f \otimes \bar g\|_2 = \|f\|_2 \|g\|_2$, and that $f \otimes \overline{g} \in L^2(X \times Y)_0^{\otimes m}$.
    Since $\Adj |_{ L^2(X \times Y)_0}$ has spectral radius at most $\rho$,
    the tensor power $\Adj^{\otimes m} |_{L^2(X \times Y)_0^{\otimes m}}$ has spectral radius (and hence operator norm) at most $\rho^m$,
    so the last expression above is bounded by $\rho^{m/2} \|f\|_2 \|g\|_2 \|h\|_2$.
\end{proof}

\begin{remark}%
  \label{rem:quasi-other}
  As stated in the introduction, while \Cref{def:quasirandom} has some nice properties (e.g., the spectral radius of $\Adj - \cU$ can be computed efficiently), it is chosen for mainly practical rather than philosophical reasons, and there are similar but qualitatively inequivalent conditions that would work equally well.

  One notable criticism of this definition is that latin squares associated to Steiner triple systems (i.e., where $X=Y=Z$ and $\sL$ contains the diagonal $\{(x,x,x) \colon x \in X\}$ and is invariant under the $S_3$-action on triples) always fail to be $\Adj$-quasirandom with parameter $\rho<1$ (since the diagonal $\{(x, x) : x \in X\}$ of $X \times X$ is a closed set for the Markov chain).
  On the other hand, a random Steiner triple system is far from having algebraic structure and presumably satisfies~\eqref{eq:quasi} for $\rho = o(1)$ with high probability as $n \to \infty$.

  One point of view is that~\eqref{eq:quasi} itself is the more natural quasirandomness condition (but harder to verify), and \Cref{def:quasirandom} is a convenient sufficient condition.
\end{remark}

\begin{proof}[Proof of \Cref{thm:dense-minor}]
  Let $A \subset [n]$ and $|A| = m$.
  By \Cref{lem:quasi-use,rem:l20},
  \[
      |\Lambda(P_A 1_S, P_A 1_S, P_A 1_S)| \leq \rho^{m/2} \|P_A 1_S\|_2^3 \leq \rho^{m/2} \|1_S\|_2^3 = \rho^{m/2} \pfrac{n!}{n^n}^3.
  \]
  Hence, for $\rho \leq 1$,
  \[
      \sum_{|A| \geq m} |\Lambda(P_A1_S, P_A1_S, P_A1_S)| \leq 2^n \rho^{m/2} \pfrac{n!}{n^n}^3.
  \]
  Taking $m = \eps n$ and $\rho$ so that $2 \rho^{\eps/2} \le 1/10$, the result follows.
\end{proof}


%

\section{Quasirandomness}%
\label{sec:quasirandomness}

In this section we will verify that two natural classes of latin squares are $\Adj$-quasirandom with parameter $o(1)$:
\begin{itemize}
    \item multiplication tables of quasirandom groups;
    \item uniformly random $n \times n$ latin squares, with high probability as $n \to \infty$.
\end{itemize}
In the case of a group we can compute the whole spectrum of $\Adj$ using representation theory.
In the case of a random latin square we will use the bound
\[
    1 + \rho^6 \leq \tr \Adj^6
\]
which holds because the spectrum of $\Adj$ is real and $6$ is even.
By interpreting $n^6 \tr \Adj^6$ as counting certain kinds of configuration in $\sL$ (and using a recent result of~\cite{KSSS}) we will show that $\tr \Adj^6 = 1 + o(1)$ with high probability,
which implies that $\rho = o(1)$.
(Using the same method one can show that $\tr \Adj^4 = 3 + o(1)$ with high probability, so $6$ is the smallest even integer that we can use for this argument.)

\subsection{Quasirandom groups}

The following proposition  shows that our quasirandomness condition generalizes the definition of a quasirandom group (see~\cite{gowers}), implying Theorem~\ref{thm:quasirandom-groups-are-quasirandom}.

\begin{proposition}
    Suppose $\sL$ is the multiplication table of a group $G$.
    Then the spectrum of $\Adj$ consists of $d^3(d+1)/2$ copies of $1/d$ and $d^3(d-1)/2$ copies of $-1/d$ for every $d$-dimensional irreducible representation of $G$, and $n^2 - \sum_{\chi \in \Irr(G)} \chi(1)^4$ zeros.
    In particular $\rho = 1/D$ where $D$ is the minimal dimension of a nontrivial representation of $G$.
\end{proposition}
\begin{proof}


Here $X = Y = Z = G$ and $\sL = \{(x, y, z) \in G^3 : xy = z\}$,
so $L^2(X \times Y) = L^2(G \times G)$ and $\Adj$ is the operator defined by
\[
    \Adj(f)(x, y) = \frac1n \sum_{z \in G} f(z y^{-1}, x^{-1} z).
\]
By representation theory, $L^2(G)$ has an orthogonal basis consisting of the functions of the form $x \mapsto \langle \rho(x) e_i, e_j\rangle$, where $\rho \colon G \to U(V)$ is an irreducible unitary representation of $G$ and $e_1, \dots, e_{\dim V}$ is an orthonormal basis of $V$.

It follows that $L^2(G \times G) \cong L^2(G) \otimes L^2(G)$ has an orthogonal basis consisting of functions of the form
\[
  f_{\rho,\rho',i,j,k,\ell}(x,y) = \bigl\langle \rho(x) e_i, e_j \bigr\rangle \bigl\langle e'_\ell, \rho'(y) e'_k \bigr\rangle
\]
where $\rho \colon G \to U(V)$ and $\rho' \colon G \to U(V')$ are two irreducible unitary representations of $G$ and $1 \le i,j \le \dim V$, $1 \le k,\ell \le \dim V'$.

To find $\Adj(f_{\rho,\rho',i,j,k,\ell})$ we recall the Schur orthogonality relation for matrix coefficients, which states that for irreducible $V$, $V'$ as above, $a,b \in V$ and $a',b' \in V'$,
\[
  \frac1n \sum_{z \in G} \langle \rho(z) a, b \rangle \langle b', \rho'(z) a' \rangle = \begin{cases} 0 &\colon (\rho,V) \ncong (\rho',V') \\ \frac1{\dim V} \langle a, a' \rangle \langle b', b \rangle &\colon (\rho,V) = (\rho',V'), \end{cases}
\]
and thereby compute
\begin{align*}
  \Adj(f_{\rho,\rho',i,j,k,\ell})(x,y)
  &= \frac1n \sum_{z \in G} \bigl\langle \rho(z) \rho(y^{-1}) e_i, e_j \bigr\rangle \bigl\langle \rho(x) e'_\ell, \rho(z) e'_k \bigr\rangle \\
  &= \begin{cases} 0 &\colon (\rho,V) \ncong (\rho',V') \\ \frac1{\dim V} \bigl\langle \rho(x) e_\ell, e_j \bigr\rangle \bigl\langle e_i, \rho(y) e_k \bigr\rangle &\colon (\rho,V) = (\rho',V')
  \end{cases}\\
  &= \begin{cases} 0 &\colon (\rho,V) \ncong (\rho',V') \\ \frac1{\dim V} f_{\rho, \rho, \ell, j, k, i}(x, y) &\colon (\rho,V) = (\rho',V').
  \end{cases}
\end{align*}
In the case $\rho \ne \rho'$ we get an eigenfunction with eigenvalue $0$.
When $\rho = \rho'$ and $i=\ell$ we get a $(1/\dim V)$-eigenfunction.
Finally when $\rho = \rho'$ and $i \ne \ell$, the functions
\[
  f_{\rho,\rho,i,j,k,\ell} \pm f_{\rho,\rho,\ell,j,k,i}
\]
are eigenfunctions of $\Adj$ with eigenvalues $\pm 1/\dim V$ respectively.


Altogether we have $d^3 + d^3(d-1)/2 = d^3 (d+1)/2$ copies of $1/d$ and $d^3(d-1)/2$ copies of $-1/d$, and the rest $0$, as claimed.

\end{proof}

\subsection{Random latin squares}

We will use a recent result of Kwan, Sah, Sawhney, and Simkin~\cite{KSSS}
on configuration counts in random latin squares.
A \emph{triple system} is a 3-uniform 3-partite hypergraph $\sH \subset X_\sH \times Y_\sH \times Z_\sH$ with vertex classes $X_\sH, Y_\sH, Z_\sH$.
The number of vertices is $v = |X_\sH| + |Y_\sH| + |Z_\sH|$ and the number of triples (hyperedges) is $e = |\sH|$.
We say $\sH$ is \emph{latin} if every pair of vertices is in at most one triple.
(A latin square of order $n$ is then a latin triple system with 3 classes of $n$ vertices and $n^2$ triples.)

Let $\sH$ be a fixed triple system.
A \emph{copy} of $\sH$ in a triple system $\sL$ is a triple of injective maps
\[
    X_\sH \to X_\sL, \qquad
    Y_\sH \to Y_\sL, \qquad
    Z_\sH \to Z_\sL
\]
which maps triples to triples.
Let $N_\sH(\sL)$ denote the number of copies of $\sH$ in $\sL$.

Let $\sB_n$ denote the random triple system $\sB_n \subset [n] \times [n] \times [n]$ in which each possible triple is present independently with probability $1/n$.
Note that $\EE [N_\sH(\sB_n)] = (1 - o(1)) n^{v-e}$
(when $\sH$ is fixed and $n$ is large).
We say $\sH$ is \emph{$\alpha$-stable}
if $\alpha \geq v-e$ and
\[
    \EE [N_\sH(\sB_n) \mid \sQ \subset \sB_n] - \EE [N_\sH(\sB_n)] = o(n^\alpha)
\]
for any latin triple system $\sQ \subset [n] \times [n] \times [n]$ with at most $n (\log n)^3$ triples.

\begin{theorem}[{\cite[Theorem~7.2]{KSSS}}]%
    \label{thm:KSSS}
    Fix an $\alpha$-stable latin triple system $\sH$ with $v$ vertices and $e$ triples.
    Let $\sL$ be a uniformly random latin square.
    Then
    \[
        N_\sH(\sL) \leq n^{v-e} + o(n^\alpha)
    \]
    with high probability as $n \to \infty$.
\end{theorem}

In order to use this theorem effectively we need a computable form of stability.
Let $\sH$ be a latin triple system.
A subset of the vertices $S \subset X_\sH \cup Y_\sH \cup Z_\sH$ is called \emph{closed} if whenever two vertices of a triple of $\sH$ is in $S$, so is the third.
The \emph{closure} $\langle S \rangle_\sH$ of a subset $S$ if the smallest closed set containing it.
If $\sF \subset \sH$ let $X_\sF, Y_\sF, Z_\sF$ denote the vertices incident with at least one member of $\sF$, and let $v(\sF) = |X_\sF| + |Y_\sF| + |Z_\sF|$ and $e(\sF) = |\sF|$.
We say $\sF \subset \sH$ \emph{generates} $\sH$ if
\[
    \langle X_\sF \cup Y_\sF \cup Z_\sF\rangle_\sH = X_\sH \cup Y_\sH \cup Z_\sH.
\]
Let
\[
    d(\sH) = \min \{ e(\sF) : \sF~\text{generates}~\sH\}.
\]

For example, if $\sH_1$ is the latin triple system shown in \Cref{fig:trAdj6},
one generating set consists of both triples containing $z_1$, one triple containing $z_3$, and one triple containing $z_5$,
and there is no smaller generating set, so $d(\sH_1) = 4$.

\begin{lemma}%
    \label{lem:stability-lemma}
    Let $\sH$ be a latin triple system with $v$ vertices and $e$ triples.
    Then $\sH$ is $\alpha$-stable provided $\alpha \geq v - e$ and
    \[
        \alpha > v - e + \max_{\emptyset \neq \sF \subset \sH} \bigl(d(\sF) - v(\sF) + e(\sF)\bigr).
    \]
\end{lemma}

\begin{remark}%
  \label{rem:subgraph}
  A much simpler model problem is the following: given a fixed graph $H$ and a random graph $G_{n,p}$, does $G$ contain $n^{v(H)} p^{e(H)} (1+o(1))$ copies of $H$ (i.e., close to the expected number) with high probability?
  The answer might be no if $H$ contains a subgraph $H'$ with much greater density than $H$ in some sense: indeed, if $n^{v(H')} p^{e(H')}= o(1)$ then with high probability $G(n,p)$ contains zero copies of $H'$, and hence of $H$.
  However, this is essentially all that can go wrong.
  The condition for $\alpha$-stability in the lemma captures a similar intuition.
\end{remark}

\begin{remark}%
    Given a triple system $\sH \subseteq X_\sH \times Y_\sH \times Z_\sH$, one can construct a partition triple $\psystem = (\pi_1,\pi_2,\pi_3) \in \Pi_\sH^3$ in the sense of Section~\ref{subsec:partitions} (i.e., the ground set has size $e(\sH)$) where two triples $(x,y,z),(x',y',z') \in \sH$ lie in the same cell of $\pi_1$ if and only if $x=x'$, and similarly for $\pi_2$ and $y=y'$, and $\pi_3$ and $z=z'$.

    The construction can be reversed (up to the issue of repeated edges).
    In other words, triple systems and partition triples are more-or-less the same objects.
    Under this analogy, the notion of closure here coincides with that in Definition~\ref{def:crank}, and $\crank(\psystem) = 2e(\sH) - d(\sH)$.

    Although using both languages is strictly speaking redundant, it is useful to keep the two notions separate, partly for minor technical reasons, but mainly because using partition systems follows our previous work in \cites{EMM-cyclic, EMM} while using triple systems follows~\cite{KSSS}.
\end{remark}

\begin{proof}[Proof of \Cref{lem:stability-lemma}]
    (Cf.~\cite[p.~15]{KSSS})
    Let $\sQ \subset [n]^3$ be a latin triple system with at most $n^{1+o(1)}$ triples.
    For a copy of $\sH$ in $\sB_n$, say one of its triples is \emph{forced}
    if it appears in $\sQ$.
    The difference
    \begin{equation}
        \label{eq:stability-difference}
        \EE[N_\sH(\sB_n) \mid \sQ \subset \sB_n] - \EE[N_\sH(\sB_n)]
    \end{equation}
    arises from copies of $\sH$ with at least one forced triple.
    Let $\sF \subset \sH$ be a nonempty subsystem
    and consider copies of $\sH$ whose forced triples are precisely the images of those in $\sF$.
    Let $\sF_0 \subset \sF$ be a generating subsystem of size $d(\sF)$.
    Because $\sQ$ satisfies the latin property, any copy of $\sF$ in $\sQ$ is determined by the image of $\sF_0$.
    Therefore the number of copies of $\sF$ in $\sQ$ is at most $|\sQ|^{|\sF_0|}$.
    There are $v - v(\sF)$ vertices of $\sH$ outside $\sF$, each with $n$ possible images in $[n]^3$,
    and the image of each of the $e - e(\sF)$ triples outside $\sF$ has probability $1/n$ (independently) of being present in $\sB_n$.
    Hence the contribution to~\eqref{eq:stability-difference} from $\sF$ is bounded by
    \[
        |\sQ|^{|\sF_0|} n^{v - v(\sF)} (1/n)^{e - e(\sF)}
        = n^{v - e + d(\sF) - v(\sF) + e(\sF) + o(1)}.
    \]
    This is $o(n^{\alpha})$ provided the stated condition is satisfied.
\end{proof}

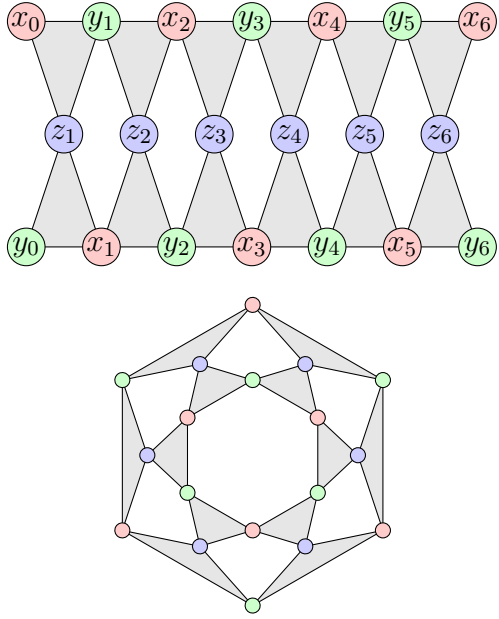
\begin{figure}
    \centering
    \begin{subfigure}{\textwidth}
        \centering
        \begin{tikzpicture}[scale=0.5]
          \def\er{0.5}
          \def\ey{3}

          \draw[fill=gray!20] (0,\ey) -- (2,\ey) -- (1,0) -- cycle;
          \draw[fill=gray!20] (2,\ey) -- (4,\ey) -- (3,0) -- cycle;
          \draw[fill=gray!20] (4,\ey) -- (6,\ey) -- (5,0) -- cycle;
          \draw[fill=gray!20] (6,\ey) -- (8,\ey) -- (7,0) -- cycle;
          \draw[fill=gray!20] (8,\ey) -- (10,\ey) -- (9,0) -- cycle;
          \draw[fill=gray!20] (10,\ey) -- (12,\ey) -- (11,0) -- cycle;
          \draw[fill=gray!20] (0,-\ey) -- (2,-\ey) -- (1,0) -- cycle;
          \draw[fill=gray!20] (2,-\ey) -- (4,-\ey) -- (3,0) -- cycle;
          \draw[fill=gray!20] (4,-\ey) -- (6,-\ey) -- (5,0) -- cycle;
          \draw[fill=gray!20] (6,-\ey) -- (8,-\ey) -- (7,0) -- cycle;
          \draw[fill=gray!20] (8,-\ey) -- (10,-\ey) -- (9,0) -- cycle;
          \draw[fill=gray!20] (10,-\ey) -- (12,-\ey) -- (11,0) -- cycle;

          \draw[fill=red!20] (0,\ey) node{$x_0$} circle (\er);
          \draw[fill=green!20] (2,\ey) node{$y_1$} circle (\er);
          \draw[fill=red!20] (4,\ey) node{$x_2$} circle (\er);
          \draw[fill=green!20] (6,\ey) node{$y_3$} circle (\er);
          \draw[fill=red!20] (8,\ey) node{$x_4$} circle (\er);
          \draw[fill=green!20] (10,\ey) node{$y_5$} circle (\er);
          \draw[fill=red!20] (12,\ey) node{$x_6$} circle (\er);

          \draw[fill=green!20] (0,-\ey) node{$y_0$} circle (\er);
          \draw[fill=red!20] (2,-\ey) node{$x_1$} circle (\er);
          \draw[fill=green!20] (4,-\ey) node{$y_2$} circle (\er);
          \draw[fill=red!20] (6,-\ey) node{$x_3$} circle (\er);
          \draw[fill=green!20] (8,-\ey) node{$y_4$} circle (\er);
          \draw[fill=red!20] (10,-\ey) node{$x_5$} circle (\er);
          \draw[fill=green!20] (12,-\ey) node{$y_6$} circle (\er);

          \draw[fill=blue!20] (1,0) node{$z_1$} circle (\er);
          \draw[fill=blue!20] (3,0) node{$z_2$} circle (\er);
          \draw[fill=blue!20] (5,0) node{$z_3$} circle (\er);
          \draw[fill=blue!20] (7,0) node{$z_4$} circle (\er);
          \draw[fill=blue!20] (9,0) node{$z_5$} circle (\er);
          \draw[fill=blue!20] (11,0) node{$z_6$} circle (\er);
      	\end{tikzpicture}
    \end{subfigure}
    \par\bigskip
    \begin{subfigure}{\textwidth}
        \centering
        \begin{tikzpicture}
            \def\er{0.1}
            \foreach \rotation in {0, 60, 120, 180, 240, 300}{
                \draw[rotate=\rotation, fill=gray!20] (30:1) -- (-30:1) -- (0:1.4) -- cycle;
                \draw[rotate=\rotation, fill=gray!20] (0:1.4) -- (30:2) -- (-30:2) -- cycle;
                \draw[rotate=\rotation, fill=blue!20] (0:1.4) circle (\er);
            }
            \foreach \rotation in {0, 120, 240}{
                    \draw[rotate=\rotation, fill=red!20] (30:1) circle (\er);
                    \draw[rotate=\rotation, fill=red!20] (-30:2) circle (\er);
                    \draw[rotate=\rotation, fill=green!20] (-30:1) circle (\er);
                    \draw[rotate=\rotation, fill=green!20] (30:2) circle (\er);
            }
        \end{tikzpicture}
    \end{subfigure}

    \caption{The chain $(x_0, y_0), \dots, (x_6, y_6)$
    and the latin triple system $\sH_1$ defined by identifying $x_0$ with $x_6$ and $y_0$ with $y_6$}%
 	\label{fig:trAdj6}
\end{figure}

Now we can show that random latin squares are $\Adj$-quasirandom with parameter $o(1)$ with high probability (\Cref{thm:random-implies-quasirandom}).
This follows from the following proposition and the bound $1 + \rho^6 \leq \tr \Adj^6$.

\begin{proposition}%
    \label{prop:random-square-tr-Adj6}
    For a uniformly random latin square $\sL$,
    \[
        \tr \Adj^6 = 1 + o(1)
    \]
    with high probability as $n \to \infty$.
\end{proposition}

\begin{proof}[Proof (computer-assisted)]
    For $(x_0, y_0) \in X \times Y$, let $(x_i, y_i)$ denote the iterates of $(x_0, y_0)$ under the Markov chain defining $\Adj$.
    Then
    \[
        \tr \Adj^6
        = \sum_{x_0, y_0} \PP\bigl((x_6, y_6) = (x_0, y_0)\bigr)
        = N / n^6,
    \]
    where $N$ is the number of configurations in $\sL$ of the form shown in \Cref{fig:trAdj6} with $x_0 = x_6$ and $y_0 = y_6$.
    We do not assume the other vertices are distinct.

    Let $\sH_1$ be the latin triple system depicted in \Cref{fig:trAdj6} and let $\sH_2, \dots, \sH_k$ (where $k$ is bounded) be all the degenerations obtainable by identifying some (like-colored) vertices and identifying triangles as necessary to preserve the latin property.

    Formally, we consider all triples of partitions $(\pi_X,\pi_Y,\pi_Z)$ where $\pi_X \in \Pi_{X_{\sH_1}}$, $\pi_Y \in \Pi_{Y_{\sH_1}}$, $\pi_Z \in \Pi_{Z_{\sH_1}}$ satisfying the following closure property: if $(x,y,z)$ and $(x',y',z')$ are two triples of $\sH_1$ and two of the pairs $(x,x')$, $(y,y')$, $(z,z')$ are in the same cell of $\pi_X$, $\pi_Y$, $\pi_Z$ respectively then so is the third.
    Number such triples of partitions $1,\dots,k$, where $1$ corresponds to three copies of the discrete partition. Then $\sH_i$ denotes the quotient hypergraph of $\sH_1$ with respect to partition $i$.

    Let $N_i = N_{\sH_i}(\sL)$.
    Then $N = N_1 + \cdots + N_k$.
    Let $v_i = v(\sH_i)$ and $e_i = e(\sH_i)$.
    Then $v_1 - e_1 = 18 - 12 = 6$.
    Now the proposition follows from \Cref{thm:KSSS}, \Cref{lem:stability-lemma}, and the following two claims:
    \begin{enumerate}
        \item $v_i - e_i \leq 5$ for each $i > 1$,
        \item $v_i - e_i + \max_{\emptyset \neq \sF \subset \sH} \bigl(d(\sF) - v(\sF) + e(\sF)\bigr) \leq 5$ for each $i \geq 1$.
    \end{enumerate}
    Indeed, provided (1) and (2) hold, \Cref{lem:stability-lemma} shows that $\sH_i$ is $6$-stable for each $i \geq 1$, so \Cref{thm:KSSS} implies that $N_i \leq n^{v_i - e_i} + o(n^6)$ with high probability for each $i$, so $N \leq (1 + o(1)) n^6$ with high probability.

    Both claims can be verified by exhaustive search.
    We find $\sH_2, \dots, \sH_k$ by starting with $\sH_1$ and iteratively identifying pairs of vertices, using breadth-first search.
    Thus we verify (1).
    Now for each $\sH_i$ we check all subsystems $\sF \subset \sH_i$ and compute $d(\sF)$ by checking all $\sF_0 \subset \sF$,
    and thus we verify (2).

    It turns out $k = 1206$, and there are $154$ distinct isomorphism classes among the degenerations $\sH_i$.
    The quantity in (2) turns out to be at most 4 in all cases except $\sH_1$, for which it is $5$.
    There are just eight degenerations $\sH_i$ (up to isomorphism) for which $v_i - e_i = 5$.
    Of these, four are just $\sH_1$ with a single pair of vertices identified (so $v_i = 17$ and $e_i = 12$).
    The other four cases are shown in \Cref{fig:degenerations}.
    These cases are therefore the dominant contributors to the error term.
\end{proof}

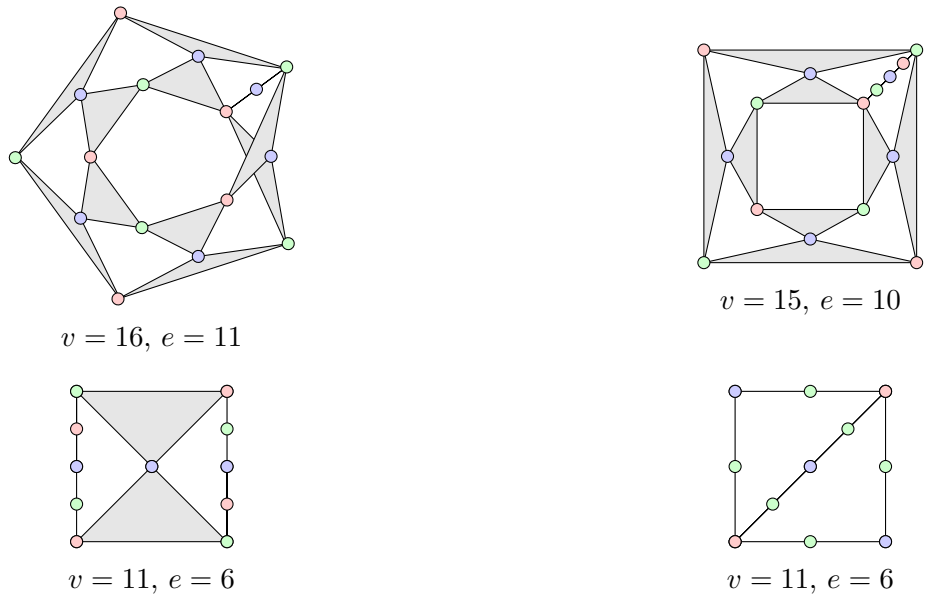
\begin{figure}
    \centering
    \begin{subfigure}{0.49\textwidth}
        \centering
        \begin{tikzpicture}
            \def\er{0.08}
            \def\angle{36.5}

            \draw[fill=gray!20] (\angle:1) -- (\angle:2) -- (\angle:1.5) -- cycle;

            \draw[fill=gray!20] (\angle:1) -- (0:1.4) -- (-\angle:2) -- cycle;
            \draw[fill=gray!20] (\angle:2) -- (0:1.4) -- (-\angle:1) -- cycle;

            \foreach \rotation in {72, 144, 216, 288}{
                \draw[rotate=\rotation, fill=gray!20] (\angle:1) -- (-\angle:1) -- (0:1.4) -- cycle;
                \draw[rotate=\rotation, fill=gray!20] (0:1.4) -- (\angle:2) -- (-\angle:2) -- cycle;
                \draw[rotate=\rotation, fill=blue!20] (0:1.4) circle (\er);
            }
            \foreach \rotation in {144, 288}{
                    \draw[rotate=\rotation, fill=red!20] (\angle:1) circle (\er);
                    \draw[rotate=\rotation, fill=red!20] (-\angle:2) circle (\er);
                    \draw[rotate=\rotation, fill=green!20] (-\angle:1) circle (\er);
                    \draw[rotate=\rotation, fill=green!20] (\angle:2) circle (\er);
            }

            \draw[fill=blue!20] (0:1.4) circle (\er);
            \draw[fill=red!20] (\angle:1) circle (\er);
            \draw[fill=green!20] (\angle:2) circle (\er);

            \draw[fill=blue!20] (\angle:1.5) circle (\er);
        \end{tikzpicture}
        \subcaption*{$v = 16$, $e = 11$}
    \end{subfigure}
    \begin{subfigure}{0.49\textwidth}
        \centering
        \begin{tikzpicture}
            \def\er{0.08}

            \draw[fill=gray!20] (45:1) -- (45:1.25) -- (45:1.5) -- cycle;
            \draw[fill=gray!20] (45:1.5) -- (45:1.75) -- (45:2) -- cycle;

            \foreach \rotation in {0, 90, 180, 270}{
                \draw[rotate=\rotation, fill=gray!20] (45:1) -- (-45:1) -- (0:1.1) -- cycle;
                \draw[rotate=\rotation, fill=gray!20] (0:1.1) -- (45:2) -- (-45:2) -- cycle;
                \draw[rotate=\rotation, fill=blue!20] (0:1.1) circle (\er);
            }
            \foreach \rotation in {0, 180}{
                    \draw[rotate=\rotation, fill=red!20] (45:1) circle (\er);
                    \draw[rotate=\rotation, fill=red!20] (-45:2) circle (\er);
                    \draw[rotate=\rotation, fill=green!20] (-45:1) circle (\er);
                    \draw[rotate=\rotation, fill=green!20] (45:2) circle (\er);
            }

            \draw[fill=green!20] (45:1.25) circle (\er);
            \draw[fill=blue!20] (45:1.5) circle (\er);
            \draw[fill=red!20] (45:1.75) circle (\er);

        \end{tikzpicture}
        \subcaption*{$v = 15$, $e = 10$}
    \end{subfigure}
    \par\bigskip
    \begin{subfigure}{0.49\textwidth}
        \centering
        \begin{tikzpicture}
            \def\er{0.08}
            \def\triples{
                0/0/0/0.5/0/1,
                0/1.5/0/2/0/1,
                0/0/2/0/1/1,
                2/0.5/2/0/2/1,
                2/2/0/2/1/1,
                2/2/2/1.5/2/1
            }
            \foreach \rx/\ry/\gx/\gy/\bx/\by in \triples
                \draw[fill=gray!20] (\rx, \ry) -- (\gx, \gy) -- (\bx, \by) -- cycle;
            \foreach \rx/\ry/\gx/\gy/\bx/\by in \triples {
                \draw[fill=red!20] (\rx, \ry) circle (\er);
                \draw[fill=green!20] (\gx, \gy) circle (\er);
                \draw[fill=blue!20] (\bx, \by) circle (\er);
            }
        \end{tikzpicture}
        \subcaption*{$v=11$, $e=6$}
    \end{subfigure}
    \begin{subfigure}{0.49\textwidth}
        \centering
        \begin{tikzpicture}
            \def\er{0.08}
            \def\triples{
                0/0/1/0/2/0,
                0/0/0/1/0/2,
                0/0/0.5/0.5/1/1,
                2/2/1/2/0/2,
                2/2/2/1/2/0,
                2/2/1.5/1.5/1/1
            }
            \foreach \rx/\ry/\gx/\gy/\bx/\by in \triples
                \draw[fill=gray!20] (\rx, \ry) -- (\gx, \gy) -- (\bx, \by) -- cycle;
            \foreach \rx/\ry/\gx/\gy/\bx/\by in \triples {
                \draw[fill=red!20] (\rx, \ry) circle (\er);
                \draw[fill=green!20] (\gx, \gy) circle (\er);
                \draw[fill=blue!20] (\bx, \by) circle (\er);
            }
        \end{tikzpicture}
        \subcaption*{$v=11$, $e=6$}
    \end{subfigure}
    \caption{Degenerations of $\sH_1$ with $v_i - e_i = 5$ and $e_i < 12$.
    Some triangles are shown flat.}%
    \label{fig:degenerations}
\end{figure}

\bibliography{refs}
\bibliographystyle{alpha}

\end{document}